\documentclass{amsart}

\usepackage{tikz-cd}

\usepackage{amssymb}
\usepackage{amsthm}
\usepackage{amscd}

\usepackage[mathscr]{eucal}
\usepackage{enumitem}

\usepackage{hyperref}

\usepackage{color}

\usepackage[backend=biber,style=alphabetic,maxnames=99]{biblatex}
\AtEveryBibitem{%
  \clearfield{doi}%
  \clearfield{url}%
  \clearfield{urldate}%
  \ifentrytype{misc}{}{\clearfield{eprint}}%
  \clearfield{archiveprefix}%
  \clearfield{primaryclass}%
  \clearfield{isbn}%
  \clearfield{issn}%
  \clearfield{note}%
  \clearlist{language}%
  \clearlist{location}%
}

\DeclareBibliographyDriver{article}{%
  \usebibmacro{bibindex}%
  \usebibmacro{begentry}%
  \printnames{author}%
  \setunit{\labelnamepunct}\newblock
  \printfield{title}%
  \newunit\newblock
  \printfield{journaltitle}%
  \setunit*{\addspace}%
  \printfield{volume}%
  \iffieldundef{number}
    {}
    {\printtext[parens]{\printfield{number}}}%
  \setunit{\addspace}%
  \printtext[parens]{\printfield{year}}%
  \setunit{\addcolon\space}%
  \printfield{pages}%
  \finentry
}

\newtheorem{Theorem}{Theorem}[section]
\newtheorem{theorem}[Theorem]{Theorem}

\newtheorem{proposition}[Theorem]{Proposition}

\newtheorem{corollary}[Theorem]{Corollary}

\newtheorem{lemma}[Theorem]{Lemma}
\newtheorem{Fact}[Theorem]{Fact}
\newtheorem{fact}[Theorem]{Fact}
\newtheorem{remark/def}[Theorem]{Remark/Definition}

\newtheorem{claim}[Theorem]{Claim}

\theoremstyle{definition}

\newtheorem{remark}[Theorem]{Remark}

\newtheorem{definition}[Theorem]{Definition}

\newtheorem{notation}[Theorem]{Notation}

\newtheorem{question}[Theorem]{Question}
\newtheorem{def/rem}[Theorem]{Definition/Remark}
\newtheorem{def/fact}[Theorem]{Definition/Fact}
\newtheorem{not/rem}[Theorem]{Notation/Remark}

\newtheorem{observation}[Theorem]{Observation}

\newsavebox{\indbin}
\savebox{\indbin}{\begin{picture}(0,0)
\newlength{\gnu}
\settowidth{\gnu}{$\smile$} \setlength{\unitlength}{.5\gnu}
\put(-1,-.65){$\smile$} \put(-.25,.1){$|$}
\end{picture}}

\newcommand{\age}{\operatorname{age}}

\newcommand{\aut}{\operatorname{Aut}}

\newcommand{\emb}{\operatorname{Emb}}
\newcommand{\SQ}{\mathbf{S}_{\mathbb{Q}}}
\newcommand{\SQP}{\mathbf{S}_{\mathbb{Q}}^+}
\newcommand{\CQ}{\mathbf{C}_{\mathbb{Q}}}

\newcommand{\len}{\operatorname{len}}
\newcommand{\fr}{Fraïssé }

\newcommand{\mb}{\mathbf}
\newcommand{\mc}{\mathcal}
\newcommand{\str}{\operatorname{str}}

\newcommand{\bin}{\operatorname{Bin}}

\newcommand{\cone}{\operatorname{Cone}}

\def\dom{\operatorname{dom}}

\def\tless#1{\lhd^{\mathbf#1}}

\def\hull#1#2{\langle#1\rangle_{#2}}

\begin{document}

\title[Ramsey properties of trees and $C$-relations]{Ramsey properties of trees and $C$-relations via the partite construction}

\date{October 13, 2026}

\author[H. Lee]{Hyoyoon Lee}
\address{Center for Nano Materials\\
 G-LAMP\\
 Sogang University\\
 Seoul, South Korea}
 \email{hyoyoonlee@sogang.ac.kr}

\begin{abstract}
Finite convexly ordered binary branching $C$-relations have been viewed as a notable exception to the broad applicability of the partite construction in structural Ramsey theory.
The absence of a proof of their Ramsey property by the partite construction was noted by Bodirsky in his 2015 survey and pointed out again by Hubi{\v{c}}ka and Kone{\v{c}}n{\'y} in their 2026 survey.
We overcome this limitation by giving proofs based on the partite construction of the known Ramsey properties of the class $\mathcal{CT}$ of finite convexly ordered treeable meet semilattices and of the classes of finite convexly ordered $C$-relations with arbitrary branching and with binary branching.
Our proofs do not rely on Milliken's tree theorem.
The key is to expand $\mathcal{CT}$ by level predicates and a predecessor function so that the locally finite subclass method applies.
We then transfer the resulting Ramsey witnesses to the corresponding reduct classes and to the classes of convexly ordered $C$-relations.
\end{abstract}



\maketitle

\tableofcontents

\section{Introduction}\label{sec:intro}

Structural Ramsey theory aims to generalize the classical Ramsey theorem for the class of finite linearly ordered sets \cite{Ram30} to general structures.
In particular, we will be interested in the Ramsey property for embeddings, rather than the Ramsey property for copies of structures.
More precisely, let $\mb A, \mb B, \mb C$ be any structures in a language $\mc L$, and let $r \ge 2$ be an integer.
$\emb(\mb A, \mb B)$ is the set of all embeddings from $\mb A$ into $\mb B$.
The notation 
$$\mb C \to (\mb B)^{\mb A}_r$$
means that for every $r$-coloring $\chi: \emb(\mb A, \mb C) \to r$, there is $\beta \in \emb(\mb B, \mb C)$ such that $\left|\chi(\beta \circ \emb(\mb A, \mb B))\right| \le 1$, where $\beta \circ \emb(\mb A, \mb B) = \{\beta \circ \alpha: \alpha \in \emb(\mb A, \mb B)\}$.
A class $\mc K$ of structures is called a {\em Ramsey class}, or has the {\em Ramsey property} if for any $\mb A, \mb B \in \mc K$ and any integer $r \ge 2$, there is $\mb C \in \mc K$ such that $\mb C \to (\mb B)^{\mb A}_r$.
We call this $\mb C$ a {\em Ramsey witness for $\mb A, \mb B \in \mc K$ (and $r$)}.

A foundational theorem of structural Ramsey theory was obtained by Ne\v{s}et\v{r}il and R\"odl \cite{NR77} and, independently, by Abramson and Harrington \cite{AH78}.
This theorem is often called the (unrestricted) Ne\v{s}et\v{r}il--R\"odl theorem and says that the class of all finite linearly ordered structures in a finite relational language has the Ramsey property.
The connection between Ramsey classes, \fr classes, and extreme amenability led to Ne\v{s}et\v{r}il's classification program of Ramsey classes \cite{Nes89, KPT2005, Nes05}.
See, for example, \cite{HK26} for an excellent survey.

One of the most powerful methods to obtain a Ramsey witness is the partite construction.
The partite construction was developed in the 70s and 80s, mainly by Ne\v{s}et\v{r}il and R\"odl.
It has many variations and generalizations, and it can be applied to find a vast range of Ramsey classes.
In particular, the locally finite subclass method is a way to apply the iterated partite construction without going through the entire process from scratch \cite{HN19}.
If $\mc R$ is a class of finite structures with the Ramsey property and $\mc K$ is a locally finite subclass of $\mc R$ with the hereditary property and the amalgamation property, and every $\mb A \in \mc K$ is irreducible, then $\mc K$ also has the Ramsey property.

To the best of the author's knowledge, no proof based on the partite construction was previously known for the Ramsey properties of the tree structures considered in this paper.
See \cite[Example 6.30 and Question 8.6]{HK26}.
If the language $\mc L$ is $\{R\}$, where $R$ is a binary relation symbol representing the edge relation, the class of finite trees in the language $\mc L$ is not closed under taking substructures.
If we take the hereditary closure of the class of finite trees in the relational language, then we get the class of finite forests, but it does not have the amalgamation property.
This motivates the study of structures on the leaves of trees, described by $C$-relations \cite[Example 2.20]{Bod15}.
Bodirsky and Piguet \cite{BP10} gave a direct proof of the Ramsey property of finite convexly ordered binary branching $C$-relations, and Bodirsky \cite[Section 5]{Bod15} noted that this result also follows from Milliken's tree theorem, while pointing out that no proof based on the partite construction was known.

Hubi\v{c}ka and Kone\v{c}n\'y \cite[Section 6.5]{HK26} describe the partite construction as a method for obtaining structural Ramsey theorems from unstructured Ramsey theorems.
\cite[Section 6.5]{HK26} also explains that some unstructured Ramsey theorems have
structural reflections that fall within the scope of Ne\v{s}et\v{r}il's classification program.
Some of the corollaries of Milliken's tree theorem can be seen as examples of such cases, and many Ramsey properties of classes of tree structures have been proved using either direct tree-specific Ramsey arguments or (corollaries of) Milliken's tree theorem.
From this perspective, it is natural to expect that the Ramsey properties of these classes would require specialized unstructured Ramsey arguments, and that the partite construction would not yield Ramsey witnesses for them without using a tree-specific unstructured Ramsey theorem as an input.

It is therefore an interesting question whether the Ramsey properties of classes of tree structures can be obtained from Ramsey witnesses produced by the partite construction, thereby bringing these classes within the scope of the partite methods used in Ne\v{s}et\v{r}il's classification program.
In this paper, we show that, through language expansions, reducts, and transfers of Ramsey witnesses, the Ramsey properties of several classes of tree structures and of structures on their leaves can be obtained from the partite construction.

\subsection*{Main results and method}

Here is a brief explanation of the procedure used in this paper.
$\mc{CT}$ is the class of finite convexly ordered treeable meet semilattices in the language $\mc L_0 = \{\sqcap, \lhd, <\}$, where $\sqcap$ is the meet operation, $\lhd$ is the tree order, and $<$ is the convex linear order.
A meet semilattice $\mb A$ is a structure with a partial order $\lhd$ such that, for every $a, b \in A$, there is a $\unlhd$-greatest lower bound $\sqcap(a, b)$.
$\mb A$ is treeable if, for every $a \in A$, the set of $\unlhd$-predecessors of $a$ is linearly ordered by $\lhd$.
It is convexly ordered by $<$ if $<$ is a linear order extending $\lhd$ and, whenever $a$ and $c$ are $\unlhd$-incomparable, $a \unlhd b$, $c \unlhd d$, and $a < c$, we have $b < d$.

In its original language $\mc L_0$, $\mc{CT}$ is not a locally finite subclass of the class of all finite $<$-linearly ordered set-valued $\mc L_0$-structures.
We extend the language $\mc L_0$ of $\mc{CT}$ by adding the level predicates $\{P_n: n \in \omega\}$ to get the expanded class $\mc{CT}_s$ in the language $\mc L_s$.
Then we further extend the language $\mc L_s$ of $\mc{CT}_s$ by adding the unary level-predecessor function $F$ and get the expanded class $\mc{CT}_F$ in the language $\mc L_F$.
We show that $\mc{CT}_s$ and $\mc{CT}_F$ are \fr classes and that the \fr limit of $\mc{CT}_F$ can be seen as a definitional expansion of the \fr limit of $\mc{CT}_s$, so that these two limits have the same automorphism group.
We also show that $\mc{CT}_F$ is a locally finite subclass of the class of all finite $<$-linearly ordered set-valued $\mc L_F$-structures.
Then Ramsey witnesses for $\mc{CT}_F$ can be obtained from the partite construction followed by completions, and we show that these witnesses can be transferred to the classes $\mc{CT}_s$ and $\mc{CT}$.
We can summarize the result as follows.
For the details, see Theorem \ref{thm:final Ramsey witness}.

\begin{theorem}\label{thm:intro 1}
    \begin{enumerate}
        \item Given $\mb A_F, \mb B_F \in \mc{CT}_F$ and $r \ge 2$, there is $\mb C_F \in \mc{CT}_F$ such that $\mb C_F \to (\mb B_F)^{\mb A_F}_r$, where $\mb C_F$ is obtained by the partite construction.
        \item Let $\mb A, \mb B \in \mc{CT}$ and $r \ge 2$.
        Starting from a natural tree structure $d^{<N}$, where natural numbers $d, N$ depend only on $\mb A, \mb B$ and $r$, we iteratively apply the partite construction for $\mc{CT}_F$ to obtain $\mb C_F$.
        Then $\mb C_F |_{\mc L_0} \to (\mb B)^{\mb A}_r$.
    \end{enumerate}
\end{theorem}

A finite convexly ordered $C$-relation can be regarded as a structure on the leaves of a convexly ordered tree in $\mc{CT}$.
$C$-relations can be axiomatized by imposing several conditions on a ternary relation $C$.
Let $\mc C^<$ be the class of finite convexly ordered $C$-relations.
For every $\mb A \in \mc C^<$, we canonically obtain $\mb A$ from a structure $\mb A_T \in \mc{CT}$ by taking its leaf structure $L(\mb A_T)$;
conversely, starting from $\mb A$, we construct a canonical tree $T(\mb A) \in \mc{CT}$ whose leaf structure is canonically isomorphic to $\mb A$.
Moreover, if $\mb A_T, \mb B_T \in \mc{CT}$ and $\alpha_T \in \emb(\mb A_T, \mb B_T)$, $\alpha_T$ naturally induces an embedding $\alpha \in \emb(L(\mb A_T), L(\mb B_T))$;
conversely, if $\mb A, \mb B \in \mc C^<$ and $\alpha \in \emb(\mb A, \mb B)$, then we can naturally define $T(\alpha) \in \emb(T(\mb A), T(\mb B))$.
Using these constructions, we can transfer a Ramsey witness in $\mc{CT}$ to a Ramsey witness for convexly ordered $C$-relations.
Then we argue that every $\mb A_T \in \mc{CT}$ can be transformed into a binary branching tree $\bin(\mb A_T)$ with the same leaf set, and that this transformation preserves embeddings of binary branching $C$-relations into the leaf structure.
Finally, we conclude that a Ramsey witness in $\mc{CT}$ can be transferred to a Ramsey witness in $\mc C^<_{\bin}$, the class of finite convexly ordered binary branching $C$-relations.
We summarize the result for convexly ordered $C$-relations as follows.
For the details, see Theorems \ref{thm:C-rel} and \ref{thm:bin C-rel}.

\begin{theorem}\label{thm:intro 2}
    \begin{enumerate}
        \item Let $\mb A, \mb B \in \mc C^<$, $r \ge 2$.
        By Theorem \ref{thm:intro 1}, we can obtain $\mb C_T \in \mc{CT}$ such that $\mb C_T \rightarrow (T(\mb B))^{T(\mb A)}_r$ by iteratively applying the partite construction.
        Then $L(\mb C_T) \in \mc C^<$ satisfies $L(\mb C_T) \rightarrow (\mb B)^{\mb A}_r$.
        \item Let $\mb A, \mb B \in \mc C^<_{\bin}$, $r \ge 2$.
        By Theorem \ref{thm:intro 1}, we can obtain $\mb C_T \in \mc{CT}$ such that $\mb C_T \rightarrow (T(\mb B))^{T(\mb A)}_r$ by iteratively applying the partite construction.
        Then $L(\bin(\mb C_T)) \in \mc C^<_{\bin}$ satisfies $L(\bin(\mb C_T)) \rightarrow (\mb B)^{\mb A}_r$.
    \end{enumerate}
\end{theorem}

\subsection*{Model-theoretic and infinitary context}

In this subsection, we briefly introduce the modeling property in model theory and big Ramsey degrees, although neither is used in the remainder of the paper.
Understanding the definition of the modeling property requires some familiarity with basic model theory.
The tree languages $\mc L_0, \mc L_{\str}$ and $\mc L_s$ in this paper also arise in the study of generalized indiscernible sequences and in model-theoretic classification theory.
Many dividing lines of classification theory can be characterized by having or omitting some combinatorial configurations satisfying several conditions, and these configurations are often in the form of trees.
The modeling property is particularly useful when studying these configurations.

Let $\mb A$ be an $\mc L_A$-structure, and let $ I =(b_a)_{a \in A}$ be a sequence of tuples of a fixed finite length in a sufficiently saturated $\mc L$-model $\mb M$.
The sequence $I$ is an {\em $\mb A$-indexed indiscernible sequence} if, for any $a_0,\ldots,a_{n-1},a_0',\ldots,a_{n-1}'\in A$, whenever $a_0\cdots a_{n-1}$ and $a_0'\cdots a_{n-1}'$ have the same quantifier-free $\mc L_A$-type, the tuples
$b_{a_0}\cdots b_{a_{n-1}}$ and $b_{a_0'}\cdots b_{a_{n-1}'}$ have the same complete first-order $\mc L$-type.
An $\mb A$-indexed sequence $J=(c_a)_{a\in A}$ of tuples of the same length in $\mb M$ is {\em (locally) based on $I$} if, for any $a_0,\ldots,a_{n-1}\in A$ and any $\mc L$-formula $\varphi(\bar x)$ of the appropriate arity, there are $a_0',\ldots,a_{n-1}'\in A$ such that $a_0\cdots a_{n-1}$ and $a_0'\cdots a_{n-1}'$ have the same quantifier-free $\mc L_A$-type and $\mb M\models \varphi(c_{a_0}\cdots c_{a_{n-1}}) \leftrightarrow \varphi(b_{a_0'}\cdots b_{a_{n-1}'})$.
We say that $\mb A$-indexed indiscernible sequences have the {\em modeling property} if, for every language $\mc L$, every sufficiently saturated $\mc L$-model $\mb M$, and every $\mb A$-indexed sequence $I$ of tuples of a fixed finite length in $\mb M$, there is an $\mb A$-indexed indiscernible sequence $J$ of tuples of the same length in $\mb M$ locally based on $I$.
If $\mb A$ is locally finite and linearly ordered, then by the results in \cite{MR3373612, Sco21}, $\age(\mb A)$ has the Ramsey property if and only if $\mb A$-indexed indiscernible sequences have the modeling property.

For example, let $\mb{CT}$ be the \fr limit of $\mc{CT}$.
The tree modeling theorem in \cite{TT12} implies that $\mb{CT}$-indexed indiscernible sequences have the modeling property.
Related tree modeling theorems were also proved in \cite{KKS14} and are used there in the proof of the dichotomy that a first-order theory has the tree property (TP, i.e., the theory is not simple) if and only if it has $\operatorname{TP}_1$ or $\operatorname{TP}_2$ \cite[Theorem 5.9]{KKS14}.
Much recent research in structural Ramsey theory concerns big Ramsey degrees of countable structures.
For a finite substructure $\mb A$ of a countable structure $\mb M$, the big Ramsey degree of $\mb A$ in $\mb M$ is the least natural number $t$, if one exists, such that for every $2 \le r \in \omega$ and every $r$-coloring $\chi: \emb(\mb A, \mb M) \to r$, there is $\mu \in \emb(\mb M, \mb M)$ such that $|\chi(\mu \circ \emb(\mb A, \mb M))| \le t$.
Recent progress includes the determination of the exact degrees for \fr limits of finitely constrained free amalgamation classes in finite binary relational languages \cite{BCDHKVZ26}.
Expansions encoding these degrees, called big Ramsey structures, provide another connection between partition properties and additional structural information.
We refer to \cite{Dob23,HZ25} for surveys.

\subsection*{Organization of the paper}

In Section \ref{subsec:set-valued}, we explicitly distinguish between ordinary structures and set-valued structures to avoid any confusion, and introduce some basic terminology.
Section \ref{subsec:Fraisse} recalls \fr theory and the Ramsey property for set-valued structures.
Section \ref{subsec:LFS} briefly explains the method of proving
Ramsey properties of classes using the notion of locally finite
subclasses, and points out that if $\mc K$ has HP and AP and is
a locally finite subclass of the class of all finite
$<$-linearly ordered set-valued structures in a fixed language,
then for any $\mb A, \mb B \in \mc K$ and any integer $r \ge 2$,
a Ramsey witness for $\mb A, \mb B$ and $r$ can be obtained from
the iterated partite construction followed by completions.
Section \ref{sec:aut transfer} explains how we can transfer Ramsey witnesses when the automorphism groups of corresponding \fr limits are the same.
Section \ref{sec:Fraisse} introduces the tree languages $\mc L_0, \mc L_{\str}, \mc L_s$ and $\mc L_F$, and the corresponding classes of tree structures
$\mc{CT}, \mc{CT}_{\str}, \mc{CT}_s$ and $\mc{CT}_F$ in this paper.
In Section \ref{sec:partite}, we prove that $\mc{CT}_F$ is a locally finite subclass of the class of all finite linearly ordered set-valued $\mc L_F$-structures.
Section \ref{sec:C-rel} obtains Ramsey witnesses for convexly
ordered $C$-relations from Ramsey witnesses in $\mc{CT}$.
Finally, Section \ref{sec:discussion} proves that
$\mc{CT}, \mc{CT}_{\str}$ and $\mc{CT}_s$ are not locally finite
subclasses of the corresponding classes of finite linearly ordered
set-valued structures, and asks whether the partite construction
can also recover the Ramsey property of the classes
$\mathcal{DT}_m$ of trees with fixed branch labels.

\section{Preliminaries}\label{sec:pre}

Throughout, $\omega$ is the set of natural numbers (and it contains $0$).
We may write $n < \omega$ instead of $n \in \omega$.
If $0 < n \in \omega$, then $n = \{0, 1, \ldots, n-1\}$ and $0 = \emptyset$.

\subsection{Set-valued structures}\label{subsec:set-valued}

For an ordinary structure $\mb M$ with universe $M$, an $n$-ary function symbol is interpreted as a map from $M^n$ to $M$.
We will also consider interpretations whose values are subsets of the universe.
We distinguish these set-valued structures from ordinary structures as follows.

\begin{definition}[Set-valued and partial structures]\label{def:set-valued strs}\ 
    \begin{enumerate}
        \item A {\em language} $\mathcal L$ consists of {\em relation symbols}, {\em function symbols}, and {\em constant symbols}, each with a specified {\em arity}.
        \item A {\em set-valued $\mathcal L$-structure} $\mathbf M$ consists of a set $M$, called its {\em universe}, together with {\em interpretations} of the symbols in $\mathcal L$ as follows.
        \begin{enumerate}
            \item If $R \in \mc L$ is an $n$-ary relation symbol, then the interpretation $R^{\mb M}$ of $R$ in $\mb M$ is a subset of $\mb M^n$.
            \item If $F \in \mc L$ is an $n$-ary function symbol where $n \ge 1$, then the interpretation $F^{\mb M}$ of $F$ in $\mb M$ is a subset of $M^{n+1}$.
            \begin{enumerate}
                \item For $\bar{a} \in M^n$, $F^{\mb M}(\bar{a}) = A$ if $A = \{a \in M: (\bar{a}, a) \in F^{\mb M}\}$.
                \item $\dom(F^{\mb M}) = \{\bar{a} \in M^n: F^{\mb M}(\bar{a}) \ne \emptyset\}$.
            \end{enumerate}
            \item If $c \in \mc L$ is a constant symbol (equivalently, a $0$-ary function symbol), then the interpretation $c^{\mb M}$ of $c$ in $\mb M$ is an element of $\mb M$.
        \end{enumerate}
        \item We call a set-valued $\mc L$-structure $\mathbf M$ an {\em $\mc L$-structure} if $|F^{\mathbf M}(\bar a)|=1$ for every function symbol $F \in \mc L$ of positive arity $n$ and every $\bar a \in M^n$.
        \item We call a set-valued $\mc L$-structure $\mathbf M$ a {\em partial $\mc L$-structure} if $|F^{\mathbf M}(\bar a)|\leq 1$ for every function symbol $F \in \mc L$ of positive arity $n$ and every $\bar a\in M^n$.
        \item We use boldface letters to denote set-valued structures, and the corresponding plain letters to denote their universes.
        \item If $\mathcal L' \subseteq \mathcal L$ and $\mathbf M$ is a set-valued $\mathcal L$-structure, its {\em $\mathcal L'$-reduct} $\mathbf M|_{\mathcal L'}$ is obtained by forgetting the interpretations of the symbols in $\mathcal L\setminus\mathcal L'$.
        We call $\mathbf M$ an {\em expansion} of $\mathbf M|_{\mathcal L'}$.
    \end{enumerate}
\end{definition}

Unless otherwise stated, all languages and universes of structures are at most countable.

\begin{remark}\label{remark:formula and satisfaction}
    We use first-order formulas and satisfaction in the usual sense for ordinary $\mathcal L$-structures; see, for example, \cite{Hodges_1993}.
    In a general set-valued or partial $\mc L$-structure, the interpretation of an $\mc L$-term need not be an element of the structure, so the satisfaction of a first-order formula by a set-valued or partial structure is not well-defined.
    Alternatively, one may replace each function symbol by a relation symbol of arity one higher, and transform a set-valued structure into an ordinary structure in the corresponding language.
    But we need to adjust the notion of substructures, so that the relation symbols that came from the function symbols in the original language are closed in the substructure.
    This relational transform is used in the recursive partite construction based on the unrestricted Ne\v{s}et\v{r}il--R\"odl theorem;
    see \cite[Appendix A.4]{HK26} and \cite{HN19}.
\end{remark}

\begin{notation}\label{nota: partial function}
    If $\mathbf A$ is a partial $\mathcal{L}$-structure and $F \in \mathcal{L}$ is a function symbol, then we will write $F^\mathbf A(a_0, \ldots, a_{n-1}) = b$ instead of $F^\mathbf A(a_0, \ldots, a_{n-1}) = \{b\}$.
\end{notation}

\begin{definition}[Structure-preserving maps]\label{def of hom and emb}
    Let $\mathcal{L}$ be a language, $\mathbf A$, $\mathbf B$ be set-valued $\mathcal{L}$-structures, $f: A \rightarrow B$.
    \begin{enumerate}
        \item $f$ is a {\em homomorphism} if
        \begin{enumerate}
            \item for every relation symbol $R \in \mathcal{L}$,
            $$\text{if }(a_0, \ldots, a_{n-1}) \in R^{\mathbf A} \text{, then } (f(a_0), \ldots, f(a_{n-1})) \in R^{\mathbf B},$$
            and
            \item for every function symbol $F \in \mathcal{L}$,
            $$f(F^{\mathbf A}(a_0, \ldots, a_{n-1})) \subseteq F^{\mathbf B}(f(a_0), \ldots, f(a_{n-1})).$$
        \end{enumerate}
        \item $f$ is an {\em embedding} if $f$ is injective, and for any relation symbol $R \in \mc L$ and function symbol $F \in \mc L$,
        \begin{enumerate}
            \item $(a_0, \ldots, a_{n-1}) \in R^{\mathbf A}$ if and only if $(f(a_0), \ldots, f(a_{n-1})) \in R^{\mathbf B}$, and
            \item $f(F^{\mb A}(a_0, \ldots, a_{m-1})) = F^{\mb B}(f(a_0), \ldots, f(a_{m-1}))$.
        \end{enumerate}
        If such an $f$ exists, we also say that $\mb A$ can be {\em embedded} in $\mb B$ or $\mb A$ {\em embeds} into $\mb B$.
        We can naturally regard $f(\mb A)$ as a substructure of $\mb B$ if $f$ is an embedding.
        \item $\operatorname{Emb}(\mathbf A, \mathbf B)$ is the set of all embeddings from $\mathbf A$ to $\mathbf B$.
        We sometimes write $\operatorname{Emb}_{\mathcal L}(\mathbf A, \mathbf B)$ to emphasize the underlying language $\mc L$.
        If $\mathbf C$ is a set-valued $\mathcal{L}$-structure and $\beta \in \emb(\mathbf B, \mathbf C)$, then $\beta \circ \emb(\mathbf A, \mathbf B) = \{\beta \circ \alpha: \alpha \in \emb(\mathbf A, \mathbf B)\}$.
        \item If the inclusion map from $A$ to $B$ is an embedding of $\mb A$ into $\mb B$, then we say that $\mathbf A$ is a {\em substructure} of $\mathbf B$ and write $\mathbf A \subseteq \mathbf B$.
        \item For $A' \subseteq A$, $\hull{A'}{\mb A}$ is the {\em hull} of $A'$ in $\mathbf A$, i.e., the smallest substructure of $\mathbf A$ containing $A'$.
        We may omit the subscript and write $\hull{A'}{}$ when the ambient structure is clear.
        Further, for ease of notation, we may write $\hull{a, b, c}{}$ instead of $\hull{\{a,b,c\}}{}$.
        \item $\mathbf A$ is {\em finitely generated} if $\mathbf A = \hull{A'}{\mb A}$ for some finite subset $A'$ of $A$.
        \item $\mb A$ is {\em locally finite} if $\hull{A'}{\mb A}$ is finite for every finite subset $A'$ of $A$.
        \item $f$ is an {\em isomorphism} if $f$ is a surjective embedding.
        \item $f$ is an {\em automorphism} of $\mathbf A$ if $f$ is an isomorphism and $\mathbf A = \mathbf B$.
        The {\em automorphism group} of $\mathbf A$ is denoted by $\aut(\mathbf A)$.
    \end{enumerate}
\end{definition}

Note that if $f: \mb A \to \mb B$ is a homomorphism and $F \in \mc L$ is a function symbol, then possibly $F^{\mb A}(\bar{a}) = \emptyset$ but $F^{\mb B}(f(\bar{a})) \ne \emptyset$.
For partial structures, this agrees with the definition in \cite[Section 1.1]{HN19}.
(In the definition of a homomorphism in \cite[Section 2.1]{HK26}, the equality for function values should be read as an inclusion.
We use this corrected formulation throughout.)

\begin{definition}[Definable relations and functions]\label{def:definable}
    Let $\mathbf M$ be an $\mathcal{L}$-structure.
    \begin{enumerate}
        \item $R \subseteq M^n$ is a ($\emptyset$-){\em definable relation in $\mathbf{M}$} if there is an $\mathcal{L}$-formula $\varphi(\bar{x})$ such that $\bar{a} \in R$ if and only if $\mathbf M \models \varphi(\bar{a})$.
        \item $F: M^n \rightarrow M$ is a ($\emptyset$-){\em definable function in $\mathbf{M}$} if there is an $\mathcal{L}$-formula $\varphi(\bar{x}, y)$ such that $F(\bar{a}) = b$ if and only if $\mathbf M \models \varphi(\bar{a}, b)$.
    \end{enumerate}
\end{definition}

Note that we can define definable constants as the definable $0$-ary functions.
That is, $c \in M$ is definable if there is an $\mathcal{L}$-formula $\varphi(y)$ such that $\{c\} = \{a \in M: \mathbf M \models \varphi(a)\}$.

\subsection{\fr theory, Ramsey classes, and the Kechris--Pestov--Todor\v{c}evi\'c correspondence}\label{subsec:Fraisse}

\begin{definition}\label{def:amalgamation class}
    Let $\mathcal{K}$ be a class of finitely generated $\mathcal{L}$-structures.
    $\mathcal{K}$ is called an {\em amalgamation class} if it is closed under isomorphisms and satisfies the following three conditions.
    \begin{enumerate}
        \item Hereditary property (HP):
        If $\mathbf B \in \mathcal{K}$ and $\mathbf A$ is a finitely generated substructure of $\mathbf B$, then $\mathbf A \in \mathcal{K}$.
        \item Joint embedding property (JEP):
        If $\mathbf A, \mathbf B \in \mathcal{K}$, then there is $\mathbf C \in \mathcal{K}$ such that $\mathbf A$ and $\mathbf B$ can be embedded in $\mathbf C$.
        \item Amalgamation property (AP):
        If $\mathbf A, \mathbf B_1, \mathbf B_2 \in \mathcal{K}$ and there are $\alpha_1 \in \emb(\mathbf A$, $\mathbf B_1)$ and $\alpha_2 \in \emb(\mathbf A, \mathbf B_2)$, then there are $\mathbf C\in \mathcal{K}$, $\beta_1 \in \emb(\mathbf B_1, \mathbf C)$ and $\beta_2 \in \emb(\mathbf B_2, \mathbf C)$ such that $\beta_1 \circ \alpha_1 = \beta_2 \circ \alpha_2$.
    \end{enumerate}
\end{definition}

For the languages without any constant symbol, we admit the empty structure.
For hereditary classes in these languages, JEP follows from AP by amalgamating over the empty structure.

\begin{definition}\label{def:age, uh}
    Let $\mathbf M$ be an $\mathcal{L}$-structure.
    \begin{enumerate}
        \item The {\em age} of $\mathbf M$, denoted by $\age(\mathbf M)$, is the class of all finitely generated $\mc L$-structures that embed into $\mb M$.
        \item The structure $\mathbf M$ is {\em ultrahomogeneous} if every isomorphism between finitely generated substructures of $\mathbf M$ extends to an automorphism of $\mathbf M$.
    \end{enumerate}
\end{definition}

Whenever convenient, we identify a member of $\operatorname{age}(\mathbf M)$ with an isomorphic copy in $\mathbf M$.

\begin{theorem}[\fr \cite{Fra86}]\label{thm: Fraisse}
    Let $\mathcal{L}$ be an at most countable language.
    \begin{enumerate}
        \item Suppose that $\mathcal{K}$ is a nonempty amalgamation class with at most countably many structures up to isomorphism, that is, $\mc K$ is a {\em \fr class}.
        Then there is an at most countable ultrahomogeneous $\mathcal{L}$-structure $\mathbf M$ such that $\age(\mathbf M) = \mathcal{K}$.
        Such a structure $\mathbf M$ is unique up to isomorphism.
        \item Conversely, if $\mathbf M$ is an at most countable ultrahomogeneous $\mathcal{L}$-structure, then $\age(\mathbf M)$ is a \fr class.
    \end{enumerate}
\end{theorem}

\begin{remark}\label{remark:set-valued Fraisse}
    We also use the terminology of Definitions \ref{def:amalgamation class} and \ref{def:age, uh} for set-valued structures, with embeddings as in Definition \ref{def of hom and emb}.
    The correspondence in Theorem \ref{thm: Fraisse} also holds between \fr classes of finite set-valued structures and at most countable locally finite ultrahomogeneous set-valued structures.
\end{remark}

\begin{definition}[\fr limits]\label{def:Fraisse limit}
    Let $\mathcal K$ be a \fr class of structures or a \fr class of finite set-valued structures.
    The corresponding at most countable ultrahomogeneous set-valued structure is called the \fr limit of $\mathcal K$.
\end{definition}

\begin{definition}[Erdős--Rado partition arrow]\label{def:partition arrow}
    Let $2 \le r \in \omega$ and $\mathbf A, \mathbf B, \mathbf C$ be set-valued $\mathcal{L}$-structures.
    $\mathbf C \rightarrow (\mathbf B)^{\mathbf A}_r$ means that for any $\chi: \emb(\mathbf A, \mathbf C) \rightarrow r$, there is $\beta \in \emb(\mathbf B, \mathbf C)$ such that $|\chi(\beta \circ \emb(\mathbf A, \mathbf B))| \le 1$.
\end{definition}

\begin{definition}[Ramsey classes]\label{def:Ramsey property}
    Let $\mathcal{K}$ be a class of set-valued $\mathcal{L}$-structures.
    $\mathcal{K}$ is a {\em Ramsey class} if for every $2 \le r \in \omega$ and $\mathbf A, \mathbf B \in \mathcal{K}$, there is $\mathbf C \in \mathcal{K}$ such that $\mathbf C \rightarrow (\mathbf B)^{\mathbf A}_r$.
    We call such a $\mathbf C$ a {\em Ramsey witness} for $\mathbf A, \mathbf B \in \mathcal{K}$ (and $r$).
    If $\mc K$ is a Ramsey class, we also say that $\mathcal{K}$ has the {\em Ramsey property}.
\end{definition}

\fr classes are closely related to Ramsey classes.

\begin{Fact}[\cite{Nes05}, \cite{Nes89}]\label{fact:Ramsey implies Fraisse}
    Let $\mc K$ be a class of finite set-valued $\mc L$-structures.
    If $\mc K$ has the joint embedding property and the Ramsey property, then it has the amalgamation property.
\end{Fact}

We recall the KPT correspondence because of its fundamental role in structural Ramsey theory.
Throughout, no background in topological dynamics is needed.
Given an $\mc L$-structure $\mb M$, we equip $\operatorname{Aut}(\mathbf M)$ with the topology of pointwise convergence on $M$ with the discrete topology.
A topological group is said to be {\em extremely amenable} if every continuous action of the group on a nonempty compact Hausdorff space has a fixed point.

We state the following form of the KPT correspondence;
see \cite[Theorem 4.8]{KPT2005} and \cite[Theorem 5.8]{HK26}.

\begin{theorem}[\cite{KPT2005}, Theorem 4.8]\label{thm:KPT}
    Let $\mathbf M$ be an at most countable locally finite ultrahomogeneous $\mc L$-structure.
    Then $\operatorname{Aut}(\mathbf M)$ is extremely amenable if and only if $\operatorname{age}(\mathbf M)$ is a Ramsey class.
\end{theorem}

\subsection{Locally finite subclasses}\label{subsec:LFS}

\begin{definition}[\cite{HN19}, Definition 2.1]\label{def:irreducible}\ 
    \begin{enumerate}
        \item Let $\mathbf A, \mathbf B_1$ and $\mathbf B_2$ be set-valued $\mathcal{L}$-structures, and $\alpha_i \in \emb(\mathbf A, \mathbf B_i)$ for each $i = 1, 2$.
        We may assume that $\mathbf A \subseteq \mathbf B_1, \mathbf B_2$ and $\alpha_1, \alpha_2$ are inclusion maps.
        Then a set-valued $\mathcal{L}$-structure $\mathbf C$ is the {\em free amalgam} of $\mathbf B_1$ and $\mathbf B_2$ over $\mathbf A$ if
        \begin{enumerate}
            \item $A = B_1 \cap B_2$,
            \item $C = B_1 \cup B_2$,
            \item $\mb B_1, \mb B_2 \subseteq \mb C$,
            \item for every relation symbol $R \in \mathcal{L}$, if $(c_0, \ldots, c_{n-1}) \in R^{\mathbf C}$, then $\{c_0$, $\ldots$, $c_{n-1}\} \subseteq B_1$ or $B_2$, and
            \item for any function symbol $F \in \mathcal{L}$, if $F^{\mathbf C}(c_0, \ldots, c_{n-1}) = D \ne \emptyset$, then $\{c_0, \ldots, c_{n-1}\} \cup D \subseteq B_1$ or $B_2$.
        \end{enumerate}
        
        \item A set-valued $\mathcal{L}$-structure $\mathbf A$ is {\em irreducible} if it is not a free amalgam of two of its proper substructures.
    \end{enumerate}
\end{definition}

Recall that we have allowed empty structures.
If a set-valued $\mc L$-structure $\mb A$ has no proper substructure, then $\mathbf A$ is trivially irreducible.
If $\mc L$ has no function symbol, then $\mb A$ is irreducible if and only if for every $a \ne b \in A$, there are $\bar{a}$ and $R \in \mc L$ such that $a, b \in \bar{a}$ and $\mb A \models R(\bar{a})$.
Otherwise, the condition that every substructure is closed under the interpretation of function symbols is important.
For example, if $\mc L = \{R, F\}$, where $R$ is a binary relation symbol and $F$ is a unary function symbol, and $\mb A$ is a set-valued $\mc L$-structure such that $A = \{a,b,c,d\}$, $R^{\mb A} = \{(b, c)\}$ and $F^{\mb A} = \{(b, a), (c, d)\}$, then $\mb A$ is irreducible.

\begin{definition}[\cite{HN19}, Definition 2.3 and \cite{HK26}, Definition 6.13]\label{def:h-e}
    Let $\mathbf A$, $\mathbf B$ be set-valued $\mathcal{L}$-structures and $f: A \rightarrow B$.
    $f$ is called a {\em homomorphism-embedding} if $f$ is a homomorphism and $f \restriction A'$ is an embedding for every irreducible substructure $\mathbf A'$ of $\mathbf A$.
    If there is a homomorphism-embedding from $\mathbf A$ to $\mathbf B$, $\mathbf B$ is called a {\em completion} of $\mathbf A$.
\end{definition}

\begin{observation}\label{obs:LFS completion}
    If $2 \le r \in \omega$, $\mb A, \mb B$ are irreducible, $\mb C \rightarrow (\mb B)^{\mb A}_r$ and $\mb C'$ is a completion of $\mb C$, then $\mb C' \rightarrow (\mb B)^{\mb A}_r$.
\end{observation}

\begin{proof}
    Let $f$ be a homomorphism-embedding from $\mb C$ to $\mb C'$.
    Let $\chi': \emb(\mb A, \mb C') \rightarrow r$ and define $\chi: \emb(\mb A, \mb C) \rightarrow r$ by $\chi(\alpha) = \chi'(f \circ \alpha)$.
    Note that $f \circ \alpha \in \emb(\mb A, \mb C')$ since $f$ is a homomorphism-embedding and $\mb A$ is irreducible.
    By $\mb C \rightarrow (\mb B)^{\mb A}_r$, there is $\beta \in \emb(\mb B, \mb C)$ such that $|\chi(\beta \circ \emb(\mb A, \mb B))| \le 1$.
    Now observe that $f \circ \beta \in \emb(\mb B, \mb C')$ and for every $\alpha \in \emb(\mb A, \mb B)$, $\chi'(f \circ \beta \circ \alpha) = \chi(\beta \circ \alpha) \in \chi(\beta \circ \emb(\mb A, \mb B))$.
\end{proof}

By letting the forbidden family $\mc F$ be empty in \cite[Theorem 6.5]{HK26} we get the following generalized unrestricted Ne\v{s}et\v{r}il--R\"odl theorem.

\begin{theorem}[\cite{HK26}, Theorem 6.5]\label{thm:NR theorem}
    Let $\mathcal{L}$ be a language possibly with both function and relation symbols, and assume that $<$ is a binary relation symbol in $\mathcal{L}$.
    Then the class of all finite $<$-linearly ordered set-valued $\mc L$-structures is a Ramsey class.
\end{theorem}

Let $\mb A, \mb B \in \mc K \subseteq \mc R$, where $\mc R$ is the class of all finite $<$-linearly ordered set-valued $\mc L$-structures.
Theorem \ref{thm:NR theorem} guarantees that $\mc R$ can be used as an `ambient' Ramsey class.
By the iterated applications of the induced partite construction starting with $\mb A$ and $\mb B$ (see Appendix of \cite{HK26}), we can obtain $\mb C$ satisfying the conditions of Definition \ref{def:LFS}(1)--(3) if $\mc K$ has HP and AP, where $n$ is the number of iterations.
As a consequence, $\mb C \rightarrow (\mb B)^{\mb A}_r$, but $\mb C$ is not necessarily in $\mc R (\supseteq \mc K)$.
However, since $\mb A, \mb B \in \mc K$ are irreducible (as they are $<$-linearly ordered), any completion $\mb C^*$ of $\mb C$ in $\mc K$ satisfies $\mb C^* \rightarrow (\mb B)^{\mb A}_r$ by Observation \ref{obs:LFS completion}.

\begin{definition}[\cite{HK26}, Definition 6.14]\label{def:LFS}
    Let $\mathcal{L}$ be any language, $\mathcal{R}$ be a class of finite set-valued $\mathcal{L}$-structures, and $\mathcal{K} \subseteq \mathcal{R}$.
    We say that $\mathcal{K}$ is a {\em locally finite subclass of $\mathcal{R}$} if for every $\mathbf B \in \mathcal{K}$ and every $\mathbf C_0 \in \mathcal{R}$, there is $n = n(\mathbf B, \mathbf C_0) \in \omega$ such that if a finite set-valued $\mc L$-structure $\mathbf C$ satisfies the following three conditions, then $\mathbf C$ has a completion in $\mathcal{K}$:
    \begin{enumerate}
        \item $\mathbf C_0$ is a completion of $\mathbf C$.
        \item Every irreducible substructure of $\mathbf C$ extends to $\mathbf B$ in $\mathbf C$.
        That is, for every irreducible substructure $\mathbf C'$ of $\mathbf C$, there is an embedding $f_{\mathbf C'}$ of $\mathbf B$ in $\mathbf C$ such that $\mathbf C' \subseteq f_{\mb C'}(\mathbf B)$.
        \item Every substructure of $\mathbf C$ of size at most $n$ has a completion in $\mathcal{K}$.
    \end{enumerate}
\end{definition}

\begin{theorem}[\cite{HN19} and \cite{HK26}, Theorem 6.15]\label{thm:LFS Ramsey transfer}
    Let $\mathcal{L}$ be a language, and let $\mathcal{R}$ be a Ramsey class of finite set-valued $\mathcal{L}$-structures.
    If
    \begin{itemize}
        \item $\mathcal{K}$ is a locally finite subclass of $\mathcal{R}$,
        \item $\mathcal{K}$ has HP and AP, and
        \item every $\mathbf A \in \mathcal{K}$ is irreducible,
    \end{itemize}
    then $\mathcal{K}$ is a Ramsey class.
\end{theorem}

Combining the preceding results gives the following constructive form of the locally finite subclass theorem.

\begin{corollary}\label{cor:LFS partite construction}
    Let $\mathcal{L}$ be any language with a binary relation symbol $<$ and $\mathcal{R}$ be the class of finite $<$-linearly ordered set-valued $\mc L$-structures.
    If
    \begin{itemize}
        \item $\mathcal{K}$ is a locally finite subclass of $\mathcal{R}$, and
        \item $\mathcal{K}$ has HP and AP,
    \end{itemize}
    then $\mathcal{K}$ is a Ramsey class.
    Furthermore, given $\mb A, \mb B \in \mc K$ and $2 \le r \in \omega$, a Ramsey witness $\mb C \in \mc K$ can be obtained by the iterated partite construction followed by completions in $\mc K$.
\end{corollary}

\section{Ramsey transfer via automorphism groups}\label{sec:aut transfer}

\begin{remark}\label{remark:def aut same}
    Let $\mathbf M$ be an $\mathcal{L}$-structure and $\mathbf M'$ be an $\mathcal{L}'$-structure, where $\mathcal{L} \subseteq \mathcal{L}'$ and $\mb M = \mb M' |_{\mc L}$.
    If the interpretation of every relation and function symbol in $\mathcal{L}' \setminus \mathcal{L}$ in $\mathbf M'$ is definable in $\mb M$ (Definition \ref{def:definable}), then $\operatorname{Aut}(\mathbf M) = \operatorname{Aut}(\mathbf M')$.
\end{remark}

\begin{proof}
    It is clear that any automorphism of $\mathbf M'$ is also an automorphism of $\mathbf M$.
    For the converse, every automorphism of $\mb M$ preserves all relations and functions definable in $\mb M$, hence preserves the interpretation of every additional symbol.
\end{proof}

By Theorem \ref{thm:KPT} and Remark \ref{remark:def aut same} we have:

\begin{corollary}\label{cor:aut same Ramsey}
    For $i = 1, 2$, let $\mb M_i$ be an at most countable locally finite ultrahomogeneous $\mathcal{L}_i$-structure.
    Assume that $\mathcal{L}_1 \subseteq \mathcal{L}_2$, and $\mathbf M_2 |_{\mathcal{L}_1} = \mathbf M_1$.
    If the interpretation of every symbol in $\mathcal{L}_2 \setminus \mathcal{L}_1$ in $\mathbf M_2$ is definable in $\mathbf M_1$, then $\age(\mathbf M_1)$ is a Ramsey class if and only if $\age(\mathbf M_2)$ is a Ramsey class.
\end{corollary}

Under the assumptions of Corollary \ref{cor:aut same Ramsey}, we may have $\{\mb A|_{\mc L_1}: \mb A \in \age(\mb M_2)\} \subsetneq \age(\mb M_1)$, since $\mc L_1$-structures need not be closed under the interpretation of function symbols in $\mc L_2 \setminus \mc L_1$.
In particular, this occurs for the structures $\mathbf S_{\mathbb Q}$ and $\mathbf C_{\mathbb Q}$ introduced in Section \ref{sec:Fraisse};
see Remark \ref{remark:SQ CQ}(4).
The following proposition transfers individual Ramsey witnesses by taking generated substructures, and its proof does not use the KPT correspondence.

\begin{proposition}\label{prop:Ramsey witness transfer}
    For $i = 1, 2$, let $\mb M_i$ be an at most countable locally finite ultrahomogeneous $\mathcal{L}_i$-structure.
    Assume that $\age(\mathbf M_1)$ has the Ramsey property, $M_1 = M_2$, and $\aut(\mathbf M_1) = \aut(\mathbf M_2)$.
    Then for any $2 \le r \in \omega$ and $\mathbf A_2, \mathbf B_2 \in \age(\mathbf M_2)$, we have $\mathbf C_2 \rightarrow (\mathbf B_2)^{\mathbf A_2}_r$, where
    \begin{itemize}
        \item $\mathbf A_1 =\ \hull{A_2}{\mb M_1}$, $\mathbf B_1 =\ \hull{B_2}{\mb M_1}$,
        \item $\mathbf C_1 \in \age(\mathbf M_1)$ is given by the Ramsey property of $\age(\mathbf M_1)$, so that $\mathbf C_1 \rightarrow (\mathbf B_1)^{\mathbf A_1}_r$, and
        \item $\mathbf C_2 =\ \hull{C_1}{\mb M_2} \in \age(\mathbf M_2)$.
    \end{itemize}
\end{proposition}

\begin{proof}
    First, we prove the following:
    For any $\alpha_1 \in \emb(\mathbf A_1, \mathbf C_1)$, $\alpha_1 \restriction A_2 \in \emb(\mathbf A_2, \mathbf C_2)$.
    By the ultrahomogeneity of $\mathbf M_1$, there is $f \in \aut(\mathbf M_1)$ extending $\alpha_1$.
    Since $\aut(\mathbf M_1) = \aut(\mathbf M_2)$, we have $f \in \aut(\mathbf M_2)$.
    Then $f \restriction A_2 = \alpha_1 \restriction A_2$ is an embedding from $\mathbf A_2$ into $\mathbf M_2$, and $\alpha_1(A_2) \subseteq C_1 \subseteq C_2$.
    Thus $\alpha_1 \restriction A_2 \in \emb(\mathbf A_2, \mathbf C_2)$.
    
    Let $\chi_2:\operatorname{Emb}(\mathbf A_2, \mathbf C_2) \rightarrow r$ be any $r$-coloring, and $\chi_1: \emb(\mathbf A_1, \mathbf C_1) \rightarrow r$ be defined by $\chi_1(\alpha_1) = \chi_2(\alpha_1 \restriction A_2)$.
    Then by $\mathbf C_1 \rightarrow (\mathbf B_1)^{\mathbf A_1}_r$, there is $\beta_1 \in \emb(\mathbf B_1, \mathbf C_1)$ such that $|\chi_1(\beta_1 \circ \emb(\mathbf A_1, \mathbf B_1))| \le 1$.
    Then by the same argument that proves $\alpha_1 \restriction A_2 \in \emb(\mathbf A_2, \mathbf C_2)$, we have $\beta_2 := \beta_1 \restriction B_2 \in \emb(\mathbf B_2, \mathbf C_2)$.
    Also, for any $\alpha_2 \in \emb(\mathbf A_2, \mathbf B_2)$, there is an extension $g \in \aut(\mathbf M_2) = \aut(\mathbf M_1)$, hence $\alpha_2 \restriction A_1 := g \restriction A_1 \in \emb(\mathbf A_1, \mathbf B_1)$ extends $\alpha_2$ (note that $A_2 \subseteq A_1$,
    but let us use the notation $\alpha_2 \restriction A_1$ for convenience).
    
    Let $\alpha_2 \in \emb(\mathbf A_2, \mathbf B_2)$ and $\alpha_2 \restriction A_1 \in \emb(\mathbf A_1, \mathbf B_1)$ be an extension of $\alpha_2$.
    Note that $\beta_2 \circ \alpha_2 = (\beta_1 \circ (\alpha_2 \restriction A_1)) \restriction A_2$, thus $\chi_2(\beta_2 \circ \alpha_2) = \chi_1(\beta_1 \circ (\alpha_2 \restriction A_1)) \in \chi_1(\beta_1 \circ \emb(\mathbf A_1, \mathbf B_1))$.
    Now we have $|\chi_2(\beta_2 \circ \emb(\mathbf A_2, \mathbf B_2))| \le |\chi_1(\beta_1 \circ \emb(\mathbf A_1, \mathbf B_1))| \le 1$.
\end{proof}

\section{\fr limits of trees}\label{sec:Fraisse}

Firstly, we introduce the languages for tree structures.
Some of the names of the languages come from the model-theoretic literature (for example, \cite[Definition 2.1]{KKS14}).
To distinguish the meet operation from logical conjunction, we denote it by $\sqcap$ rather than $\wedge$.

\begin{definition}\label{def:tree languages}
    \begin{enumerate}
        \item $\mathcal{L}_0 = \{\sqcap, \lhd, <\}$, where $\sqcap$ is a binary function symbol, and $\lhd$ and $<$ are binary relation symbols.
        \item $\mathcal{L}_{\mathrm{str}} = \{\sqcap, \lhd, <, <_{\mathrm{len}}\}$ is the {\em strong Shelah language}, where $<_{\operatorname{len}}$ is a binary relation symbol.
        \item $\mathcal{L}_s = \{\sqcap, \lhd, <, (P_n)_{n \in \omega}\}$ is the {\em Shelah language}, where each $P_n$ is a unary relation symbol.
        \item $\mathcal{L}_s^+ = \mathcal{L}_{\operatorname{str}} \cup \{P_n: n \in \omega\} = \mathcal{L}_s \cup \{<_{\operatorname{len}}\}$.
        \item $\mathcal{L}_F = \{\sqcap, \lhd, <, (P_n)_{n \in \omega}, F\}$, where $F$ is a unary function symbol.
    \end{enumerate}
\end{definition}

Note that $\mathcal{L}_0$ is the smallest language in Definition \ref{def:tree languages}.
We will mainly work with tree-like structures and interpret these symbols in the natural way:
$\sqcap$ is the meet (= $\lhd$-infimum = $\lhd$-greatest lower bound) of two elements of a tree, $\lhd$ is the tree order, and $<$ is a linear order on the elements of the tree such that $\lhd \subseteq <$ ($<$ extends $\lhd)$.

In fact, in the Shelah language, $<$ is usually denoted by $<_{\operatorname{lex}}$ in the model-theoretic literature (lex is an abbreviation for `lexicographic').
Because $<$ can be interpreted as a linear order which is not necessarily convex (= lexicographic) in our context, we drop the subscript `lex'.

Now we fix terminology for semilattices and tree structures.
Many of these terms come from \cite{Sok15}.

\begin{definition}\label{def:semilattice}
    Let $A$ be a set partially ordered by $\lhd$ and $\sqcap$ be a binary function from $A^2$ to $A$.
    \begin{enumerate}
        \item $(A, \sqcap, \lhd)$ is a {\em meet semilattice} if for any $a, b \in A$, the $\lhd$-infimum $c \in A$ of $a$ and $b$ exists, and $\sqcap(a, b) = c$.
        \item A meet semilattice $(A, \sqcap, \lhd)$ is {\em treeable} if 
        \begin{enumerate}
            \item for every $a \in A$, $\{b \in A: b \unlhd a\}$ is linearly ordered by $\unlhd$, and
            \item if $A \ne \emptyset$, then there is a $\lhd$-minimum element, called the {\em root}.
        \end{enumerate}
        If $(A, \sqcap, \lhd)$ is treeable, then $\lhd$ may be called a {\em tree order}.
        \item Suppose that $(A,\sqcap,\lhd)$ is a meet semilattice and that $<$ is a linear order on $A$ extending $\lhd$.
        $(A, \sqcap, \lhd, <)$ is {\em convexly ordered} by $<$ if for every $b \in A$, the cone $\{a \in A: b \unlhd a\}$ is a $<$-interval.
        More precisely, if $b_1$ and $b_2$ are $\unlhd$-incomparable and $b_1 < b_2$, then $a_1 < a_2$ for any $a_1 \unrhd b_1$ and $a_2 \unrhd b_2$.
        \item A treeable meet semilattice $(A, \sqcap, \lhd)$ is {\em leveled by $(P_n)_{n \in \omega}$} if every element of $A$ belongs to exactly one $P_n$, and if $a \lhd b$, $a \in P_n$, $b \in P_m$, then $n < m$.
        We may omit $(P_n)_{n \in \omega}$ and simply say $(A, \sqcap, \lhd)$ is leveled.
    \end{enumerate}
\end{definition}

It is pretty verbose to say that `$\mathbf A$ is a convexly ordered leveled treeable meet semilattice'.
Instead, we will call such an $\mathbf A$ a {\em Shelah tree}.

\begin{remark}\label{remark:Shelah tree}
    Let $(A, \sqcap, \lhd, <, (P_n)_{n \in \omega})$ be a nonempty Shelah tree.
    \begin{enumerate}
        \item The root of $(A, \sqcap, \lhd, <, (P_n)_{n \in \omega})$ may not have the level $P_0$.
        We can start from any level, say, $P_n$, and the immediate $\lhd$-successors of the root may have levels strictly greater than $P_{n+1}$, and their levels are not necessarily the same.
        \item $(A, \sqcap, \lhd, <, (P_n)_{n \in \omega})$ can be naturally regarded as an $\mathcal{L}_s$-structure $\mathbf A_s$.
        \item The immediate $\lhd$-predecessor of an element in a finite Shelah tree need not be one level below it.
    \end{enumerate}
\end{remark}

Again, our notation for classes of finite tree structures mostly comes from \cite{Sok15}, with subscripts indicating the additional structure.

\begin{definition}\label{def: class of finite Shelah trees}
    \begin{enumerate}
        \item $\mathcal{R}_s$ is the class of all finite $<$-linearly ordered set-valued $\mathcal{L}_s$-structures.
        \item $\mathcal{ET}_s (\subseteq \mathcal{R}_s)$ is the class of all finite $<$-linearly ordered leveled treeable meet semilattices, where $<$ extends $\lhd$.
        \item $\mathcal{CT}_s (\subseteq \mathcal{ET}_s)$ is the class of all finite Shelah trees.
        \item $\mc{CT} = \{\mb A|_{\mc L_0}: \mb A \in \mc{CT}_s\}$ is the class of all finite convexly ordered treeable meet semilattices in the language $\mathcal L_0$.
    \end{enumerate}
\end{definition}

We cannot hope that $\mc{ET}_s$ is a Ramsey class by the following remark and Fact \ref{fact:Ramsey implies Fraisse}.

\begin{remark}\label{remark: K_< not Fraisse}
    $\mc{ET}_s$ has JEP, but not AP.
\end{remark}

\begin{proof}
    We first show that $\mc{ET}_s$ does not have AP.
    Define $\mathbf A, \mathbf B_1, \mathbf B_2 \in \mc{ET}_s$ as follows.
    \begin{itemize}
        \item $A = \{a, b, c\}$, $\mathbf A \models P_0(a) \wedge P_1(c) \wedge P_2(b) \wedge (\sqcap(b, c) = a) \wedge (a < c < b)$.
        \item $B_1 = A \cup \{d\}$, $\mathbf A \subseteq \mathbf B_1$, $\mathbf B_1 \models P_1(d) \wedge (a \lhd d \lhd b) \wedge (a < d < c < b)$.
        \item $B_2 = A \cup \{e\}$, $\mathbf A \subseteq \mathbf B_2$, $\mathbf B_2 \models P_1(e) \wedge (a \lhd e \lhd b) \wedge (a < c < e < b)$.
    \end{itemize}
    Let $\alpha_1$ and $\alpha_2$ be the inclusion maps from $\mathbf A$ to $\mathbf B_1$ and $\mathbf B_2$, respectively.
    Suppose that there are $\mathbf C \in \mc{ET}_s$ and $\beta_i \in \operatorname{Emb}(\mathbf B_i, \mathbf C)$ for $i = 1, 2$, such that $\beta_1 \circ \alpha_1 = \beta_2 \circ \alpha_2$.
    We may assume that $\beta_1$ is an inclusion map.
    Since $\beta_2(e)$ and $d$ are $\lhd^{\mathbf C}$-less than $b$ and $\mathbf C$ is treeable, $\beta_2(e)$ and $d$ must be $\unlhd^{\mathbf C}$-comparable.
    Note that $\mathbf C \models P_1(d) \wedge P_1(\beta_2(e))$, thus the only possibility is $\beta_2(e) = d$.
    But then $\mb C \models d < c < \beta_2(e) = d$, a contradiction.

    On the other hand, the class $\mathcal{ET}_s$ has the joint embedding property:
    Place two structures above a common $P_0$-root, identifying any existing $P_0$-roots, and leave all other elements distinct.
\end{proof}

\begin{definition}\label{def:Q tree}
    \begin{enumerate}
        \item $\mathbb Q^{< \omega}$ is the set of all functions from some $n = \{0, \ldots, n-1\} \subseteq \omega$ to $\mathbb Q$.
        \item $\SQP = (\mathbb Q^{< \omega}, \sqcap, \lhd, <, <_{\operatorname{len}}, (P_n)_{n \in \omega})$ is an $\mathcal{L}_s^+$-structure with the following interpretations.
        Let $f, f_1, f_2 \in \mathbb Q^{< \omega}$.
        \begin{enumerate}
            \item Let $m = \max\{k: f_1 \restriction k = f_2 \restriction k \text{ and } k \le \min\{\dom(f_1), \dom(f_2)\}\}$.
            Define $\sqcap^{\SQP}(f_1, f_2) = f_1 \restriction m = f_2 \restriction m$.
            \item $\SQP \models f_1 \unlhd f_2$ if and only if $f_1 \subseteq f_2$.
            Equivalently, $\SQP \models f_1 \unlhd f_2$ if and only if $\sqcap^{\SQP}(f_1, f_2) = f_1$.
            \item Let $m = \operatorname{dom}(\sqcap^{\SQP}(f_1, f_2))$.
            $\SQP \models f_1 < f_2$ if and only if $\SQP \models f_1 \lhd f_2$, or $f_1$ and $f_2$ are $\unlhd$-incomparable and $f_1(m) < f_2(m)$.
            \item $\SQP \models f_1 <_{\operatorname{len}} f_2$ if and only if $\operatorname{dom}(f_1) < \operatorname{dom}(f_2)$.
            \item $\SQP \models P_m(f)$ if and only if $\operatorname{dom}(f) = m$.
        \end{enumerate}
        \item $\SQ = \SQP |_{\mathcal{L}_s} = (\mathbb Q^{< \omega}, \sqcap, \lhd, <, (P_n)_{n \in \omega})$.
        \item $\CQ = (\mathbb Q^{< \omega}, \sqcap, \lhd, <, (P_n)_{n \in \omega}, F)$ is an $\mathcal{L}_F$-structure such that $\CQ |_{\mathcal{L}_s} = \SQ$ and $F^{\CQ}(f) = f \restriction (m - 1)$ if $m = \operatorname{dom}(f) > 0$, $F^{\CQ}(f) = f$ otherwise.
    \end{enumerate}
\end{definition}

\begin{remark}\label{remark:SQ CQ}
    \begin{enumerate}
        \item $\operatorname{age}(\SQ)$ is the class of finite Shelah trees, $\mathcal{CT}_s$.
        \item $F^{\CQ}$ is definable in $\SQ$ (Definition \ref{def:definable}).
        Explicitly, $\CQ \models F(a) = b$ if and only if $\SQ \models (b \lhd a \wedge \forall z(z \lhd a \rightarrow z \unlhd b)) \vee (\not\exists z(z \lhd a) \wedge a = b)$.
        \item For every $n > 0$, $\CQ \models \forall xy((F(x) = y \wedge x \ne y \wedge P_n(x)) \rightarrow P_{n-1}(y))$.
        \item Since $\mathbf A \in  \operatorname{age}(\CQ)$ is closed under $F^{\CQ}$, for any $a \in A$, $\{n: \mathbf A \models \exists x(P_n(x) \wedge x \unlhd a)\}$ is an initial segment of $\omega$.
        Thus $\{\mathbf A |_{\mathcal{L}_s}: \mathbf A \in \age(\CQ)\} \subsetneq \age(\SQ) = \mathcal{CT}_s$.
    \end{enumerate}
\end{remark}

\begin{proposition}\label{prop:s homo}
    $\SQ = (\mathbb Q^{< \omega}, \sqcap, \lhd, <, (P_n)_{n \in \omega})$ is ultrahomogeneous.
    Consequently, $\operatorname{age}(\SQ) = \mathcal{CT}_s$ is a Fraïssé class.
\end{proposition}

\begin{proof}
    Let $\mathbf A,\mathbf B$ be finite substructures of $\mathbf S_{\mathbb Q}$, and let $f:\mathbf A\to\mathbf B$ be an $\mathcal L_s$-isomorphism.
    Let $\mb A^* = \{a \upharpoonright n: a \in A,\ 0 \leq n \leq \operatorname{dom}(a)\}$ and define $\mb B^*$ similarly.
    These sets are finite and closed under initial segments, hence we can regard them as substructures of $\mathbf S_{\mathbb Q}$.
    Define $f^*: \mb A^* \rightarrow \mb B^*$ by $f^*(a\upharpoonright n)=f(a)\upharpoonright n$.
    Then $f^*:\mathbf A^*\to\mathbf B^*$ is an $\mathcal L_s$-isomorphism extending $f$.
    
    We now extend $f^*$ to an automorphism $\bar f$ of $\mathbf S_{\mathbb Q}$, proceeding level by level.
    Put $\bar f(\emptyset)=\emptyset$.
    Suppose that $\bar f(v)$ has been defined for every $v \in S_{\mathbb Q}$ such that $\SQ \models P_k(v)$, where $k \le n$, with $\bar f(v)=f^*(v)$ whenever $v\in A^*$.
    For any $v$ with $\SQ \models P_n(v)$, write $\lhd^{\SQ}$-successors of $v$ with level $P_{n+1}$ as $v^\frown\langle q\rangle$, where $q\in\mathbb Q$.
    The images prescribed by $f^*$ determine a finite increasing map between the rational labels of the $\lhd^{\SQ}$-immediate successors of $v$ and those of $\bar f(v)$.
    Extend this finite map to an order automorphism $\sigma_v$ of $(\mathbb Q,<)$, and define
    $$\bar f(v^\frown\langle q\rangle) = \bar f(v)^\frown\langle\sigma_v(q)\rangle$$
    for each $q \in \mathbb Q$.
    This inductive construction extends $f^*$ to 
    $\bar f\in\operatorname{Aut}(\mathbf S_{\mathbb Q})$.
\end{proof}

\begin{definition}\label{def:class of Shelah trees with function}
    \begin{enumerate}
        \item $\mathcal{R}_F$ is the class of all $<$-linearly ordered finite set-valued $\mathcal{L}_F$-structures.
        \item $\mathcal{CT}_F = \operatorname{age}(\CQ) (\subseteq \mathcal{R}_F)$ is the class of finite convexly ordered leveled treeable meet semilattices with an immediate $\lhd$-predecessor (or, more precisely, level predecessor) function $F$.
        
        Equivalently, a nonempty finite $\mathcal L_F$-structure $\mathbf A$ belongs to $\mathcal{CT}_F$ if and only if $\mathbf A|_{\mathcal L_s}$ is a Shelah tree, its root $r$ belongs to $P_0^{\mathbf A}$, $F^{\mathbf A}(r)=r$, and every $a\in P_n^{\mathbf A}$ with $n>0$ has immediate $\lhd^{\mathbf A}$-predecessor $F^{\mathbf A}(a)\in P_{n-1}^{\mathbf A}$.
        \item $\mc{CT}_s^+ = \age(\SQP)$.
        \item $\mc{CT}_{\str} = \{\mb A |_{\mc L_{\str}}: \mb A \in \mc{CT}^+_s\}$.
    \end{enumerate}
\end{definition}

\begin{remark}\label{remark:CQ and SQP}
    Every $\mb A_s \in \mc{CT}_s = \age(\SQ)$ has a unique expansion $\mb A_s^+$ to $\mc{CT}_s^+ = \age(\SQP)$, since for any $a, b \in A_s$, $\SQP \models a <_{\len} b$ if and only if there are $n, m \in \omega$ such that $\SQ \models P_n(a) \wedge P_m(b)$ and $n < m$.
    Furthermore, for any $\mathbf A_s,\mathbf B_s\in\mathcal{CT}_s$, $\emb_{\mc L_s}(\mb A_s, \mb B_s) = \emb_{\mc L_s^+}(\mb A_s^+, \mb B_s^+)$.
\end{remark}

\begin{proposition}\label{prop:F homo}
    $\CQ = (\mathbb Q^{< \omega}, \sqcap, \lhd, <, (P_n)_{n \in \omega}, F)$ is ultrahomogeneous.
    Consequently, $\operatorname{age}(\CQ) = \mathcal{CT}_F$ is a Fraïssé class.
\end{proposition}

\begin{proof}
    Let $\mathbf A,\mathbf B$ be finite substructures of $\mathbf C_{\mathbb Q}$, and let $f:\mathbf A\to\mathbf B$ be an $\mathcal L_F$-isomorphism.
    $\mb A|_{\mc L_s}$ and $\mb B|_{\mc L_s}$ are finite substructures of $\SQ$, and $f$ is an $\mc L_s$-isomorphism between them.
    By Proposition \ref{prop:s homo}, $f$ extends to some $\bar f\in\operatorname{Aut}(\mathbf S_{\mathbb Q})$.
    Since $F^{\mathbf C_{\mathbb Q}}$ is definable in $\mathbf S_{\mathbb Q}$ (Remark \ref{remark:SQ CQ}(2)), we have $\bar f\in\operatorname{Aut}(\mathbf C_{\mathbb Q})$ (Remark \ref{remark:def aut same}).
\end{proof}

\begin{corollary}\label{lemma: CQ SQ Ramsey}
    $\mc{CT}_s$ is a Ramsey class if and only if $\mc{CT}_F$ is a Ramsey class.
\end{corollary}

\begin{proof}
    By Remark \ref{remark:def aut same} and Remark \ref{remark:SQ CQ}(2), $\aut(\SQ) = \aut(\CQ)$.
    Thus by Corollary \ref{cor:aut same Ramsey}, Propositions \ref{prop:s homo} and \ref{prop:F homo}, $\age(\SQ) = \mc{CT}_s$ is a Ramsey class if and only if $\age(\CQ) = \mc{CT}_F$ is a Ramsey class.
\end{proof}

\section{Proof by the partite construction}\label{sec:partite}

We first prove that $\mathcal{CT}_F$ is a locally finite subclass of $\mathcal R_F$, and then transfer the resulting Ramsey witnesses to $\mathcal{CT}_s$ and $\mathcal{CT}_s^+$.

\begin{observation}\label{obs:basic}
    Let $\mathbf{B}, \mathbf C, \mathbf C_0$ be as in Definition \ref{def:LFS}, $n = 1$, $\mathcal{K} = \mathcal{CT}_F$ and $\mathcal{R} = \mathcal{R}_F$.
    Then we have the following easy observations, and {\em we may use them without further mention}.
    Here, the $F^{\mathbf C}$-closure of $a$ means the smallest subset of $C$ containing $a$ and closed under $F^{\mathbf C}$.
    \begin{enumerate}
        \item If $\sqcap^{\mathbf C}(a, b)$ is not empty, then $\hull{a,b}{\mb C}$ is irreducible.
        \item For any $a, b \in C$, $|\sqcap^{\mathbf C}(a, b)| \le 1$.
        Thus we will write $\sqcap^{\mathbf C}(a, b) = c$ instead of $\sqcap^{\mathbf C}(a, b) = \{c\}$.
        \item $\sqcap^{\mathbf C}$ is idempotent and commutative.
        That is, for any $a, b \in C$, $\sqcap^{\mathbf C}(a, a) = a$ and $\sqcap^{\mathbf C}(a, b) = \sqcap^{\mathbf C}(b, a)$.
        \item For any distinct $a, b \in C$, if $\sqcap^{\mathbf C}(a, b) = c$, then $c \unlhd^{\mathbf C} a, b$.
        \item For any $a \in C$, there is a unique $n \in \omega$ such that $a \in P_n^{\mathbf C}$.
        \item For any $a \in C$, $|F^{\mathbf C}(a)| = 1$.
        \item $a \unlhd^{\mathbf C} b$ if and only if $a$ is in the $F^{\mathbf C}$-closure of $b$.
        Also, if $k$ is the least number such that $a$ can be obtained from $k$ applications of $F^{\mathbf C}$ to $b$ and $b \in P_n^{\mathbf C}$, then $a \in P_{n - k}^{\mathbf C}$.
        \item For any $a \in C$, $\hull{a}{\mb C}$ is the $F^{\mathbf C}$-closure of $a$ in $\mathbf C$, and it is the same as $\{b \in C: b \unlhd^{\mathbf C} a\}$.
        \item For any $a \in C$, the set of $\lhd^{\mathbf C}$-predecessors of $a$ is linearly ordered by $\lhd^{\mathbf C}$.
        \item Suppose $a, b \in C$ are {\em $\lhd^{\mathbf C}$-connected}, that is, there is a sequence $a = a_0, a_1, \ldots, a_n = b$ such that for every $i < n$, $a_i \lhd^{\mathbf C} a_{i + 1}$ or $a_{i+1} \lhd^{\mathbf C} a_i$.
        Then there is $c \in C$ such that $c\unlhd^{\mathbf C}a, b$.
        \item If $a <^{\mb C} b$, then $\hull{a, b}{\mb C}$ is irreducible, hence there is $c\in C$ such that $c = \sqcap^{\mb C}(a, b)$.
        \item $<^\mathbf C$ is acyclic.
        \item $\lhd^{\mathbf C} \subseteq <^{\mathbf C}$.
    \end{enumerate}
\end{observation}

\begin{proof}
    \begin{enumerate}
        \item Let $\mathbf C' := \hull{a,b}{\mb C}$.
        Suppose that $\mathbf C'$ is a free amalgam of two proper substructures $\mathbf C'_1$ and $\mathbf C'_2$ of $\mathbf C'$ for a contradiction.
        Since $\sqcap^{\mathbf C}(a, b) \ne \emptyset$, by the definition of free amalgamation, $a$ and $b$ belong to the same $C'_i$.
        Thus $\mathbf C'_i$ is a substructure containing $a$ and $b$, hence $\mb C'_i = \mb C'$, a contradiction.
        \item Suppose that $\sqcap^{\mathbf C}(a, b)$ is not empty.
        Then $\hull{a, b}{\mb C}$ is irreducible by (1), so by Definition \ref{def:LFS}(2), there is an embedding from $\hull{a, b}{\mb C}$ into $\mathbf B \in \mathcal{CT}_F$.
        Since the value of $\sqcap^{\mathbf B}$ is always a singleton, (2) follows.
        \item[(3)-(7)] Note that for every $a\in C$, the substructure $\langle a\rangle_{\mathbf C}$ is irreducible.
        These items easily follow from Definition \ref{def:LFS}(2).
        \item[(8)] It easily follows from (7) (and other observations).
        \item[(9)] Follows from (7).
        \item[(10)] Follows from (9).
        \item[(11)] The irreducibility of $\hull{a,b}{\mb C}$ is clear.
        By Definition \ref{def:LFS}(2), $\hull{a,b}{\mb C}$ embeds into $\mb B \in \mc{CT}_F$, thus there is $c = \sqcap^{\mb C}(a, b)$.
        \item[(12)] Recall that there is a homomorphism-embedding from $\mathbf C$ into $\mathbf C_0$ by Definition \ref{def:LFS}(1).
        Since every homomorphism-embedding is a homomorphism and $<^{\mathbf C_0}$ is acyclic, $<^\mathbf C$ is acyclic.
        \item [(13)] This follows from Definition \ref{def:LFS}(2). \qedhere
    \end{enumerate}
\end{proof}

\begin{theorem}\label{thm:F LFS}
    $\mathcal{CT}_F$ is a locally finite subclass of $\mathcal{R}_F$.
\end{theorem}

\begin{proof}
    Fix $\mathbf B\in\mathcal{CT}_F$ and $\mathbf C_0\in\mathcal R_F$, and put $n(\mathbf B,\mathbf C_0)=1$.
    Let $\mathbf C$ be a finite set-valued $\mathcal L_F$-structure satisfying the conditions in Definition \ref{def:LFS}(1)--(3).
    We show that $\mathbf C$ has a completion in $\mathcal{CT}_F$.
    Let $C^0, \ldots, C^{l-1}$ be an enumeration of all $\tless{C}$-connected components ($\lhd^{\mb C}$-connectedness is defined in Observation \ref{obs:basic}(10)) of $\mathbf C$. 
    Note that if $a \in C^i$ and $b \in C^j$ where $i < j$, then there is no relation between $a$ and $b$ at all; $\sqcap^{\mb C}(a, b) = \emptyset$ and $a, b$ are $<^{\mb C}$-incomparable (Observation \ref{obs:basic}(11)).

    \begin{claim}\label{claim: meet well defined}
        For any $i < l$ and any $a, b \in C^i$, there is a $\lhd^{\mathbf C}$-greatest lower bound $c$ of $a, b$.
        Furthermore, if $\sqcap^{\mb C}(a, b) = d$ for some $d$, then $d = c$.
    \end{claim}
    
    \begin{proof}[Proof of Claim \ref{claim: meet well defined}]
        For any $a, b \in C^i$, since $\mb C^i$ is $\lhd^{\mathbf C}$-connected, there is $w \in C^i$ such that $w \unlhd^{\mathbf C} a, b$ (Observation \ref{obs:basic}(10)).
        As the $\lhd^{\mathbf C}$-predecessors of $a$ are linearly ordered (Observation \ref{obs:basic}(9)), there is a $\lhd^{\mathbf C}$-greatest such $w$, say, $c \in C^i$.
        Assume that $\sqcap^{\mb C}(a, b) = d$ for some $d \in C^i$.
        Note that $\hull{a,b}{\mb C} (\subseteq \mathbf C^i)$ is irreducible and all of the $\lhd^{\mathbf C}$-predecessors of $a$ and $b$ are in $\hull{a, b}{\mb C}$.
        By Definition \ref{def:LFS}(2), there is an embedding $e:\ \hull{a,b}{\mb C}\ \rightarrow \mathbf B$.
        Since $c\unlhd^{\mathbf C}a,b$, we have $e(c)\unlhd^{\mathbf B} \sqcap^{\mathbf B}(e(a),e(b))=e(d)$.
        Thus $c\unlhd^{\mathbf C}d$.
        On the other hand, $d\unlhd^{\mathbf C}a,b$, so the choice of $c$ gives $d\unlhd^{\mathbf C}c$.
        Hence $c=d$.
    \end{proof}

    \begin{claim}\label{claim:convex}
        For any $i < l$, we can define a convex linear order $<^{\mb C^i}$ on $C^i$ that extends $<^{\mb C}$ on $C^i$.
        That is, there is $<^{\mb C^i}$ on $C^i$ such that $<^{\mb C} \cap {(C^i)}^2 \subseteq <^{\mb C^i}$, $<^{\mb C^i}$ is a linear order on $C^i$, and it satisfies the following condition:
        If $a, b \in C^i$ are $\unlhd^{\mb C}$-incomparable, $a <^{\mb C^i} b$, $a \unlhd^{\mb C} c$ and $b \unlhd^{\mb C} d$, then $c <^{\mb C^i} d$.
    \end{claim}

    \begin{proof}[Proof of Claim \ref{claim:convex}]
        For each $c\in C^i$, let $\operatorname{im}(c)$ be the set of immediate $\lhd^{\mathbf C}$-successors of $c$.
        We will first linearly order $\operatorname{im}(c)$ for each $c\in C^i$, and then define $<^{\mb C^i}$ lexicographically using these local orders.

        Let $<^{\mb C^i}$ on $\operatorname{im}(c)$ be any linear order extending the restriction of $<^{\mathbf C}$ to $\operatorname{im}(c)$ and satisfying the following condition:
        For distinct $a,b\in\operatorname{im}(c)$, we have $a<^{\mb C^i}b$ if there are $x,y\in C$ such that $x\unrhd^{\mathbf C}a$, $y\unrhd^{\mathbf C}b$, and $x<^{\mathbf C}y$.
        To see that such a linear extension exists, suppose that these conditions hold.
        By Observation \ref{obs:basic}(8) and (11), $\langle x,y\rangle_{\mathbf C}$ contains $a,b$ and is irreducible, hence embeds into $\mathbf B\in\mathcal{CT}_F$ by Definition \ref{def:LFS}(2).
        Since $a,b$ are $\unlhd^{\mathbf C}$-incomparable, the convexity of $<^{\mathbf B}$ implies $a<^{\mathbf C}b$.
        Thus all the required $<^{\mb C^i}$-comparisons already belong to the restriction of $<^{\mathbf C}$ to $\operatorname{im}(c)$.
        Since this relation is acyclic by Observation \ref{obs:basic}(12), such a linear extension exists.

        For any distinct $a,b\in C^i$, if $a\lhd^{\mathbf C}b$, then we let $a<^{\mb C^i}b$.
        If $a,b$ are $\unlhd^{\mathbf C}$-incomparable, then by Claim \ref{claim: meet well defined} there is $c\in C^i$, a $\lhd^{\mathbf C}$-greatest lower bound of $a$ and $b$.
        There are unique $x,y\in\operatorname{im}(c)$ such that $x\unlhd^{\mathbf C}a$ and $y\unlhd^{\mathbf C}b$.
        We have already defined $<^{\mb C^i}$ locally on $\operatorname{im}(c)$; we let $a<^{\mb C^i}b$ if and only if $x<^{\mb C^i}y$.

        To see that $<^{\mb C^i}$ extends $<^{\mathbf C}$ on $C^i$, suppose that $a<^{\mathbf C}b$.
        If $a,b$ are $\unlhd^{\mathbf C}$-comparable, then $a\lhd^{\mathbf C}b$ by Observation \ref{obs:basic}(12)--(13), so $a<^{\mb C^i}b$.
        Otherwise, $a<^{\mb C^i}b$ follows from the condition on the local order at their $\lhd^{\mb C}$-greatest lower bound.
        Furthermore, we have defined $<^{\mb C^i}$ lexicographically, so it is a convex linear order.
    \end{proof}
    
    We are ready to complete the partial $\mathcal{L}_F$-structure $\mathbf C$ to an $\mathcal{L}_F$-structure $\mathbf D \in \mathcal{CT}_F$.
    As an intermediate step, we define a partial $\mathcal{L}_F$-structure $\mathbf C_*$ as follows.
    The universe of $\mathbf C_*$ is $C$.
    For every $i < l$ and $a, b \in C^i$, let $\sqcap^{\mb C_*}(a, b) = c$, where $c$ is the $\lhd^{\mathbf C}$-greatest lower bound of $a, b$.
    By Claim \ref{claim: meet well defined}, $\sqcap^{\mb C_*}$ is well-defined.
    Again, for every $i < l$ and $a, b \in C^i$, $a <^{\mb C_*} b$ if and only if $a <^{\mb C^i} b$, where $<^{\mb C^i}$ is a convex linear order obtained from Claim \ref{claim:convex}.
    For other symbols, let $\mathbf C_* |_{\mathcal{L}_F \setminus \{\sqcap, <\}} = \mathbf C |_{\mathcal{L}_F \setminus \{\sqcap, <\}}$.
    The identity map from $\mathbf C$ to $\mathbf C_*$ is a homomorphism-embedding.
    Note that $\mathbf C_*$ is a disjoint union of structures in $\mathcal{CT}_F$, i.e., $\hull{C^i}{\mb C_*} \in \mc{CT}_F$ for each $i < l$.

    Finally, we define the $\mc L_F$-structure $\mathbf D$.
    The universe $D$ is formed by identifying the roots of $\mathbf C_*$.
    Formally, let $r_i$ be the element of the component on $\mb C^i$ with level $P_0$ and $d \notin \bigcup_{i < l} C^i$.
    $D = \{d\} \cup \bigcup_{i < l} C^i\setminus\{r_i\}$.
    Let $q: C \to D$ be the map such that $q \restriction C^i\setminus\{r_i\}$ is the identity map on $C^i\setminus\{r_i\}$ and $q(r_i) = d$ for every $i < l$.
    Let $d$ be the root of $\mathbf D$, which has level $P_0$.
    We naturally extend $\sqcap^{\mb D}$:
    For any pair $i \ne j < l$, $a \in C^i \setminus \{r_i\}$ and $b \in C^j \setminus \{r_j\}$, let $\sqcap^{\mathbf D}(a, b) = d$.
    Extend $<^{\mathbf C_*}$ so that for any $i < j < l$, $a \in C^i \setminus \{r_i\}$ and $b \in C^j \setminus \{r_j\}$, $a <^{\mb D} b$.
    The restriction of $q$ to each component $\mb C^i$ is an $\mc L_F$-embedding.
    Since every irreducible substructure of $\mathbf C_*$ is contained in one $\mb C^i$, $q$ is a homomorphism-embedding.
    Consequently, $\mathbf D$ is also a completion of $\mathbf C$.
\end{proof}

\begin{corollary}\label{cor:Shelah Ramsey}
    Let $2 \le r \in \omega$.
    \begin{enumerate}
        \item $\mathcal{CT}_F = \age(\CQ)$ is a Ramsey class.
        Given $\mb A_F, \mb B_F \in \mc{CT}_F$, the Ramsey witness $\mb C_F$ can be obtained from the iterated partite construction followed by a completion
        in $\mathcal{CT}_F$.
        \item $\mathcal{CT}_s = \age(\SQ)$ is a Ramsey class.
        Given $\mb A_s, \mb B_s \in \mc{CT}_s$, the Ramsey witness $\mb C_s$ can be obtained from (1) and Proposition \ref{prop:Ramsey witness transfer}:
        $\mb C_s =\ \hull{C_F}{\SQ}$, where $\mb C_F \rightarrow (\mb B_F)^{\mb A_F}_r$, $\mb A_F =\ \hull{A_s}{\CQ}, \mb B_F =\ \hull{B_s}{\CQ}$, and $\mb C_F$ is obtained from (1).
        \item $\mathcal{CT}_s^+ = \age(\SQP)$ is a Ramsey class.
        Given $\mb A_s^+, \mb B_s^+ \in \mc{CT}_s^+$, the Ramsey witness $\mb C_s^+$ can be obtained from (2):
        $\mb C_s^+ \in \mc{CT}_s^+$ is the unique expansion of $\mb C_s$, where $\mb C_s \rightarrow (\mb B_s)^{\mb A_s}_r$, $\mb B_s^+ |_{\mc L_s}= \mb B_s$ and $\mb A_s^+ |_{\mc L_s}= \mb A_s$.
    \end{enumerate}
\end{corollary}

\begin{proof}
    \begin{enumerate}
        \item $\mc{CT}_F$ is a \fr class (Proposition \ref{prop:F homo}).
        Thus the statement follows from Theorem \ref{thm:F LFS} and Corollary \ref{cor:LFS partite construction}.
        \item $F^{\CQ}$ is definable in $\SQ$ (Remark \ref{remark:SQ CQ}(2)), hence $\aut(\CQ) = \aut(\SQ)$ (Remark \ref{remark:def aut same}).
        Since $\SQ$ is ultrahomogeneous (Proposition \ref{prop:s homo}), we can use Proposition \ref{prop:Ramsey witness transfer} to transfer the Ramsey witness.
        \item Every element of $\mc{CT}_s$ has a unique expansion in $\mc{CT}_s^+$ and $\emb(\mb A_s^+, \mb B_s^+) = \emb(\mb A_s^+ |_{\mc L_s}, \mb B_s^+ |_{\mc L_s})$ (Remark \ref{remark:CQ and SQP}). \qedhere
    \end{enumerate}
\end{proof}

For the passage to $\mathcal{CT}_{\mathrm{str}}$ and $\mathcal{CT}$, we adapt ideas from \cite{TT12,KKS14}.
Proposition \ref{prop:str Ramsey} removes the level predicates by a finite iteration of Corollary \ref{cor:Shelah Ramsey}(3), followed by an application of the finite Ramsey theorem to the set of levels.
Lemma \ref{lemma:lex = len} and Proposition \ref{prop:convex treeable semilattice Ramsey} then remove level comparison by choosing expansions in which it coincides with the convex order.
These arguments transform the Ramsey witnesses obtained by Corollary \ref{cor:Shelah Ramsey} into Ramsey witnesses for the reduct classes.
Combined with the construction above, they yield the partite-construction-based procedure of obtaining Ramsey witnesses for $\mathcal{CT}$, summarized in Theorem \ref{thm:final Ramsey witness}.

\begin{definition}\label{def:height}
    \begin{enumerate}
        \item Let $\mathbf A \in \mc{CT}$ and $a_0 \in A$ be the $\lhd^{\mathbf A}$-minimum element (i.e., root) of $\mathbf A$.
        The {\em height of $a$ in $\mathbf A$}, denoted by $h_{\mathbf A}(a)$, is recursively defined:
        \begin{itemize}
            \item $h_{\mathbf A}(a_0) = 0$.
            \item If $h_{\mathbf A}(b) = m$, $b \lhd^{\mathbf A} a$ and there is no $c \in A$ such that $b \lhd^{\mathbf A} c \lhd^{\mathbf A} a$, then $h_{\mathbf A}(a) = m + 1$.
        \end{itemize}
        The {\em height of $\mathbf A$}, denoted by $h(\mb A)$, is the maximum of $\{h_{\mathbf A}(a): a \in A\}$.
        \item If $\mathbf A \in \mathcal{CT}_{\operatorname{str}}$ and $a_0 \in A$ is the root of $\mathbf A$, then again, the {\em height of $a$ in $\mathbf A$}, denoted by $h_{\mathbf A}(a)$, is recursively defined:
        \begin{itemize}
            \item $h_{\mb A}(a_0) = 0$.
            \item If $h_{\mb A}(b) = m$, $b <_{\len}^{\mb A} a$ and there is no $c \in A$ such that $b <_{\len}^{\mb A} c <_{\len}^{\mb A} a$, then $h_{\mb A}(a) = m + 1$.
        \end{itemize}
        The {\em height of $\mathbf A$}, denoted by $h(\mb A)$, is the maximum of $\{h_{\mathbf A}(a): a \in A\}$.
        \item For $\mb A \in \mathcal{CT}_s^+$ and $a \in A$, $h_{\mb A}(a)$ and $h(\mb A)$ are computed in $\mathbf A|_{\mc L_{\str}}$.
        \item For a nonempty $\mathbf A\in\mathcal{CT}$, the branching number $\operatorname{br}(\mathbf A)$ is the maximum number of immediate $\lhd^{\mathbf A}$-successors of a vertex of $A$.
        We use the same notation for expansions of $\mathbf A$, since this number depends only on the tree order.
    \end{enumerate}
\end{definition}

For $\mb A \in \mc{CT}_{\operatorname{str}}$ and $a, b \in A$, note that $h_{\mb A}(a)$ is not necessarily the same as $h_{\mb A|_{\mc L_0}}(a)$, and $h_{\mb A}(a) < h_{\mb A}(b)$ if and only if $a <_{\len}^{\mb A} b$.

\begin{proposition}\label{prop:str Ramsey}
    $\mathcal{CT}_{\operatorname{str}}$ is a Ramsey class.
\end{proposition}

\begin{proof}
    Recall that $\mathcal{L}_s^+ = \mathcal{L}_{\operatorname{str}} \cup \{P_n: n \in \omega\} = \mathcal{L}_s \cup \{<_{\operatorname{len}}\}$.
    Let $2 \le r \in \omega$, $\mathbf A, \mathbf B \in \mathcal{CT}_{\operatorname{str}}$, $k = h(\mb A) + 1$, and $n = h(\mb B) + 1$.
    By (finite) Ramsey's theorem on linear orders, there is $N \in \omega$ such that $N \rightarrow (n)^k_r$.
    Let ${N \choose k}$ denote the set of all subsets of $N$ of size $k$.
    For each $S = \{s_0 < \cdots < s_{k-1}\}\in {N \choose k}$, let $\mathbf A_S$ be the $\mathcal{L}_s^+$-expansion of $\mathbf A$, such that every element of $\mathbf A$ with height $i$ has level $P_{s_i}$.
    Likewise, for each $n$-subset $T = \{t_0 < \cdots < t_{n-1}\}$ of $N$, let $\mathbf B_T$ be the corresponding $\mathcal{L}_s^+$-expansion.

    Put $d = \max\{\operatorname{br}(\mathbf A), \operatorname{br}(\mathbf B), 1\}$.
    Let $\mathbf D_0$ be the natural $\mathcal L_s^+$-structure
    on $d^{<N}$, where $P_j$ consists of the sequences of length $j$, so that $\mb D_0 \in \mc{CT}_s^+$.
    Then every $\mathbf A_S$ with $S\in {N\choose k}$ and every $\mathbf B_T$ with $T\in {N\choose n}$ embeds into $\mathbf D_0$.
    Enumerate ${N \choose k}$ by $S_0, \ldots, S_{m-1}$.
    Since $\mathcal{CT}_s^+$ is a Ramsey class, we can inductively find $\mathbf D_i \in \mathcal{CT}_s^+$ such that $\mathbf D_{i+1} \rightarrow (\mathbf D_i)^{\mathbf A_{S_i}}_r$ for each $i < m$.
    Note that we may assume that the set of levels of $\mathbf D_m$ is $\{P_k: k < N\}$.
    We claim that $\mathbf C = \mathbf D_m |_{\mathcal{L}_{\operatorname{str}}}$ satisfies $\mathbf C \rightarrow (\mathbf B)^{\mathbf A}_r$.
    
    Let $\chi: \operatorname{Emb}_{\mathcal{L}_{\operatorname{str}}}(\mathbf A, \mathbf C) \rightarrow r$ be any $r$-coloring.
    For each $\alpha \in \operatorname{Emb}_{\mathcal{L}_{\operatorname{str}}}(\mathbf A, \mathbf C)$, there is $S_i$ such that the set of levels of $\alpha(A)$ in $\mathbf D_m$ is exactly $\{P_k: k \in S_i\}$.
    Thus for each $i < m$, we can define $\chi_i: \operatorname{Emb}_{\mathcal{L}_s^+}(\mathbf A_{S_i}, \mathbf D_m) \rightarrow r$ by $\chi_i(\alpha) = \chi(\alpha)$ (note that $\operatorname{Emb}_{\mathcal{L}_s^+}(\mathbf A_{S_i}, \mathbf D_m) \subseteq \operatorname{Emb}_{\mathcal{L}_{\operatorname{str}}}(\mathbf A, \mathbf C)$).
    Then by the construction of $\mathbf D_m$, there is $\delta \in \operatorname{Emb}_{\mathcal{L}_s^+}(\mathbf D_0, \mathbf D_m)$ such that $|\chi_i(\delta \circ \operatorname{Emb}_{\mathcal{L}_s^+}(\mathbf A_{S_i}, \mathbf D_0))| \le 1$ for all $i < m$.
    Let $\chi': {N \choose k} \rightarrow r$ be given by $\chi'(S_i) =  \chi_i(\delta \circ \alpha)$ for some (any) $\alpha \in \operatorname{Emb}_{\mathcal{L}_s^+}(\mathbf A_{S_i}, \mathbf D_0)$.
    Then by the choice of $N$, there is $T \in {N \choose n}$ such that $\left|\chi'\left({T \choose k}\right)\right| \le 1$.
    Finally, by the choice of $\mathbf D_0$, there is $\beta_T \in \operatorname{Emb}_{\mathcal{L}_s^+}(\mathbf B_T, \mathbf D_0)$.
    Observe that for any $\alpha_1, \alpha_2 \in \operatorname{Emb}_{\mathcal{L}_{\operatorname{str}}}(\mathbf A, \mathbf C)$, the corresponding $\mathcal{L}_s^+$-expansions of $\alpha_1(\mathbf A)$ and $\alpha_2(\mathbf A)$ in $\mathbf D_m$ are isomorphic as $\mathcal{L}_s^+$-structures if and only if the sets of levels of $\alpha_1(\mathbf A)$ and $\alpha_2(\mathbf A)$ in $\mathbf D_m$ are the same.
    For any $\alpha \in \operatorname{Emb}_{\mathcal{L}_{\operatorname{str}}}(\mathbf A, \mathbf B)$, let $S_i \in {T \choose k} \subseteq {N \choose k}$ be the set of levels occupied by $\beta_T \circ \alpha(A)$.
    Then we have $\beta_T \circ \alpha \in \emb_{\mc L_s^+}(\mb A_{S_i}, \mb D_0)$ and $\chi(\delta \circ \beta_T \circ \alpha) = \chi_i(\delta \circ \beta_T \circ \alpha) = \chi'(S_i) \in \chi'\left({T \choose k}\right)$.
\end{proof}

\begin{lemma}\label{lemma:lex = len}
    For every $\mathbf A \in \mathcal{CT}$, there is $\mathbf A' \in \mc{CT}_{str}$ such that $\mb A'|_{\mc L_0} = \mb A$ and for any $a, b \in A' (= A)$, $a <^{\mathbf A'} b$ if and only if $a <_{\operatorname{len}}^{\mathbf A'} b$.
\end{lemma}

\begin{proof}
    Enumerate $A$ as $a_0<^{\mathbf A}\cdots<^{\mathbf A}a_{n-1}$.
    Expand $\mathbf A$ to an $\mathcal L_s$-structure $\mathbf A_s$ by setting
    \[ P_i^{\mathbf A_s} =
    \begin{cases}
    \{a_i\}, & i<n,\\
    \emptyset, & i\ge n.
    \end{cases} \]
    Then $\mathbf A_s$ is a Shelah tree.
    Let $\mathbf A_s^{+}\in\mathcal{CT}_s^+$ be its unique expansion, and put $\mathbf A' =\mathbf A_s^{+}|_{\mathcal L_{str}}$.
    Then $<_{\mathrm{len}}^{\mathbf A'}=<^{\mathbf A'}$, as required.
\end{proof}

\begin{proposition}\label{prop:convex treeable semilattice Ramsey}
    $\mathcal{CT}$ is a Ramsey class.
\end{proposition}

\begin{proof}
    Let $2 \le r \in \omega$, $\mathbf A, \mathbf B \in \mathcal{CT}$, and let $\mathbf A', \mathbf B' \in \mathcal{CT}_{str}$ be given by Lemma \ref{lemma:lex = len}.
    By Proposition \ref{prop:str Ramsey}, there is $\mathbf C^* \in \mathcal{CT}_{str}$ such that $\mathbf C^* \rightarrow (\mathbf B')^{\mathbf A'}_r$.
    We claim that $\mathbf C:= \mathbf C^* |_{\mathcal{L}_0}$ satisfies $\mathbf C \rightarrow (\mathbf B)^{\mathbf A}_r$.

    Let $\chi: \operatorname{Emb}_{\mathcal{L}_0}(\mathbf A, \mathbf C) \rightarrow r$, and define $\chi': \operatorname{Emb}_{\mathcal{L}_{\operatorname{str}}}(\mathbf A', \mathbf C^*) \rightarrow r$ by $\chi'(\alpha) = \chi(\alpha)$.
    Then there is $\beta' \in \operatorname{Emb}_{\mathcal{L}_{\operatorname{str}}}(\mathbf B', \mathbf C^*)$ such that $|\chi'(\beta' \circ \operatorname{Emb}_{\mathcal{L}_{\operatorname{str}}}(\mathbf A', \mathbf B'))| \le 1$.
    Since $<_{\mathrm{len}}^{\mathbf A'}=<^{\mathbf A'} = <^{\mb A}$
    and $<_{\mathrm{len}}^{\mathbf B'}=<^{\mathbf B'} = <^{\mb B}$, $\operatorname{Emb}_{\mathcal L_0}(\mathbf A,\mathbf B) = \operatorname{Emb}_{\mathcal L_{str}}(\mathbf A',\mathbf B')$.
    Thus for any $\alpha \in \operatorname{Emb}_{\mathcal{L}_0}(\mathbf A, \mathbf B)$, $\chi(\beta' \circ \alpha) = \chi'(\beta' \circ \alpha) \in \chi'(\beta' \circ \operatorname{Emb}_{\mathcal{L}_{\operatorname{str}}}(\mathbf A', \mathbf B'))$.
\end{proof}

\begin{theorem}\label{thm:final Ramsey witness}
    By the proofs of Propositions \ref{prop:str Ramsey} and \ref{prop:convex treeable semilattice Ramsey}, given $2 \le r \in \omega$ and $\mathbf A, \mathbf B \in \mathcal{CT}$, we can find $\mathbf C \in \mathcal{CT}$ satisfying $\mathbf C \rightarrow (\mathbf B)^{\mathbf A}_r$ as follows;
    in the following item (4), note that $\mb D_1, \ldots, \mb D_m \in \mc{CT}^+_s$ are obtained from the iterated partite construction (Corollary \ref{cor:Shelah Ramsey}(1)--(3)).
    \begin{enumerate}
        \item Expand $\mathbf A, \mathbf B$ to $\mathbf A', \mathbf B' \in \mathcal{CT}_{\operatorname{str}}$, so that for any $a, b \in A$, $a <^{\mathbf A'} b$ if and only if $a <^{\mathbf A'}_{\len} b$ (Lemma \ref{lemma:lex = len}), and likewise for $\mathbf B'$.
        \item Let $|A| = k (=h(\mb A') + 1)$, $|B| = n (=h(\mb B') + 1)$, and $N \rightarrow (n)^k_r$.
        Expand $\mathbf A'$ to $\mathbf A'_{S_0}, \ldots, \mathbf A'_{S_{m-1}} \in \mathcal{CT}_s^+$, where $\{S_i: i < m\} = {N \choose k}$.
        Similarly, expand $\mathbf B'$ to $\{\mathbf B'_T: T \in {N \choose n}\}$.
        \item Let $\mathbf D_0$ be the natural $\mc L_s^+$-structure with universe $d^{< N}$, where 
        $$d = \max\{\operatorname{br}(\mathbf A), \operatorname{br}(\mathbf B),1 \}.$$
        \item Using Corollary \ref{cor:Shelah Ramsey}(3), choose $\mathbf D_1,\ldots,\mathbf D_m\in\mathcal{CT}_s^+$ such that $\mathbf D_{i+1} \rightarrow (\mathbf D_i)^{\mathbf A'_{S_i}}_r$ for every $i < m$.
        \item $\mathbf C := \mathbf D_m |_{\mathcal{L}_0}$ satisfies $\mathbf C \rightarrow (\mathbf B)^{\mathbf A}_r$.
    \end{enumerate}
\end{theorem}

\section{Ramsey properties of convexly ordered $C$-relations}\label{sec:C-rel}

In this section, we use the subscript $T$ for tree structures and embeddings between them, such as $\mathbf A_T$ and $\alpha_T$, to distinguish them from the leaf structures ($C$-relations) and their embeddings.

$C$-relations give an axiomatic setting of the ternary relation induced on the leaves of a tree.
We use the canonical-tree representation in \cite[Proposition 2.4]{DM24}, a version of the representation theorem of Adeleke and Neumann \cite{AN98}.
Starting from the Ramsey witnesses provided by Theorem \ref{thm:final Ramsey witness}, we first obtain Ramsey witnesses for finite convexly ordered $C$-relations with arbitrary branching.
We then pass to binary branching by refining the witnessing tree without changing its leaves.
The binary branching Ramsey theorem is known \cite[Theorem 1]{BP10}, \cite[Theorem 5.1]{Bod15};
the purpose here is to derive it by an explicit transformation of the witnesses obtained from the partite construction.

\begin{definition}\label{def:C-rel}
    A ternary relation $C$ on a set $M$ is said to be a {\em $C$-relation} if for any $a, b, c, d \in M$,
    \begin{enumerate}
        \item $C(a, b, c) \leftrightarrow C(a, c, b)$,
        \item $C(a, b, c) \rightarrow \neg C(b, a, c)$,
        \item $C(a, b, c) \rightarrow C(a, d, c) \vee C(d, b, c)$, and
        \item $a \ne b \rightarrow C(a, b, b)$.
    \end{enumerate}
    A $C$-relation is called {\em binary branching} if it additionally satisfies
    \begin{enumerate}
        \item[(5)] $a \ne b \wedge a \ne c \wedge b \ne c \rightarrow C(a, b, c) \vee C(b, a, c) \vee C(c, a, b)$.
    \end{enumerate}
    A $C$-relation is {\em convexly ordered} by $<$ if $<$ is a linear order and satisfies
    \begin{enumerate}
        \item[(6)] $C(a, b, c) \rightarrow ((a < b) \wedge (a < c)) \vee ((b < a) \wedge (c < a))$.
    \end{enumerate}
\end{definition}

\begin{definition}\label{def:language of C-rel}
    \begin{enumerate}
        \item $\mc L_C = \{C\}$ and $\mc{L}^<_C = \{C, <\}$, where $C$ is a ternary relation symbol and $<$ is a binary relation symbol.
        \item If $\mb M$ is an $\mc L_C$-structure satisfying (1)--(4) of Definition \ref{def:C-rel}, then we say that $\mb M$ is a {\em C-relation with arbitrary branching}.
        For brevity, we may simply say that $\mb M$ is a {\em $C$-relation}.
        \item If $\mb M$ is a $C$-relation satisfying Definition \ref{def:C-rel}(5), then it is called a {\em binary branching $C$-relation}.
        \item If $\mb M$ is an $\mc L^<_C$-structure, $\mb M |_{\mc L_C}$ is a $C$-relation, $<^{\mb M}$ is a linear order and $\mb M$ satisfies Definition \ref{def:C-rel}(6), then $\mb M$ is called a {\em convexly ordered $C$-relation}.
    \end{enumerate}
\end{definition}

\begin{definition}\label{def:leaves of trees}
    Let $\mb A_T \in \mc{CT}$.
    \begin{enumerate}
        \item $L(A_T)$ is the set of leaves ($\lhd^{\mb A_T}$-maximal elements) of $\mb A_T$.
        \item $L(\mb A_T)$ is an $\mc{L}^<_C$-structure on $L(A_T)$, defined as follows.
        For any $a, b, c \in L(A_T)$,
        \begin{enumerate}
            \item $L(\mb A_T) \models C(a, b, c)$ if and only if $\mb A_T \models \sqcap(a, b) = \sqcap(a, c) \lhd \sqcap(b, c)$.
            \item $L(\mb A_T) \models a < b$ if and only if $\mb A_T \models a < b$.
        \end{enumerate}
    \end{enumerate}
\end{definition}

A direct verification shows that for any $\mb A_T \in \mc{CT}$, $L(\mb A_T)$ is a convexly ordered $C$-relation on $L(A_T)$.
Also, for any $a, b, c \in L(\mb A_T)$, $\mb A_T \models \sqcap(a, b) = \sqcap(a, c) \lhd \sqcap(b, c)$ if and only if $\mb A_T \models \sqcap(a, \sqcap(b, c)) \lhd \sqcap(b, c)$.

\begin{definition}\label{def:class of C-rel}
    \begin{enumerate}
        \item $\mc C$ is the class of all finite C-relations.
        \item $\mc C^<$ is the class of all finite convexly ordered $C$-relations.
        \item $\mc C_{\bin}$ is the class of all finite binary branching C-relations.
        \item $\mc C^<_{\bin}$ is the class of all finite convexly ordered binary branching C-relations.
    \end{enumerate}
\end{definition}

For each $\mb A \in \mc C^<$, we can canonically define $T(\mb A) \in \mc{CT}$ as follows.
For details, see \cite[Proposition 2.4]{DM24} and \cite{AN98}.
Define a preorder $\preccurlyeq$ on $A^2$:
\begin{center}
    $(a, b) \preccurlyeq (c, d)$ if and only if ($\neg C(c, a, b)$ and $\neg C(d, a, b)$).
\end{center}
Define an equivalence relation $\equiv_C$ on $A^2$ by
\begin{center}
    $(a, b) \equiv_C (c, d)$ if and only if 
    $\left((a, b) \preccurlyeq (c, d) \text{ and } (c, d) \preccurlyeq (a, b)\right)$.
\end{center}
Let $T(A) := A^2 / \equiv_C$ and denote $(a, b) / \equiv_C$ by $[a, b]$.
Define
\begin{center}
    $[a, b] \lhd^{T(\mb A)} [c, d]$ if and only if $(a, b) \prec (c, d)$,
\end{center}
and $\sqcap^{T(\mb A)}([a, b], [c, d])$ is their $\unlhd^{T(\mb A)}$-greatest lower bound.
Then $T(\mb A)$ can be regarded as a meet semilattice (Definition \ref{def:semilattice});
$L(T(A)) = \{[a, a]: a \in A\}$ and $[a, b] = \sqcap^{T(\mb A)}([a, a], [b, b])$.

For each $a \in T(A)$, let $X_a := \{x \in A: a \unlhd^{T(\mb A)} [x, x]\}$.
If $a, b$ are $\unlhd^{T(\mb A)}$-comparable, define $a<^{T(\mathbf A)}b$ if and only if $a\lhd^{T(\mathbf A)}b$.
If $a, b$ are not $\unlhd^{T(\mb A)}$-comparable, by the convexity of $<^{\mb A}$, $X_a$ and $X_b$ are disjoint nonempty $<^{\mathbf A}$-intervals.
Thus exactly one of the following holds.
\begin{itemize}
    \item $x <^{\mb A} y$ for all $(x, y) \in X_a \times X_b$; in this case, define $a <^{T(\mb A)} b$.
    \item $x >^{\mb A} y$ for all $(x, y) \in X_a \times X_b$; let $a >^{T(\mb A)} b$.
\end{itemize}
Then $T(\mb A) \in \mc{CT}$.
$\iota_{\mb A}: \mb A \rightarrow L(T(\mb A))$ defined by
$$\iota_{\mb A}(a) = [a, a]$$
is an isomorphism.

For any $\alpha \in \emb(\mb A, \mb B)$, by \cite[Proposition 2.4]{DM24} and because $\alpha$ preserves the convex order on $\mb A$, there is $T(\alpha) \in \emb(T(\mb A), T(\mb B))$ such that
\begin{center}
    $T(\alpha)([a, b]) = [\alpha(a), \alpha(b)]$.
\end{center}

\begin{observation}\label{obs:tree*}
    For each $\mb A_T \in \mc{CT}$, define the following map $*_{\mb A_T}:A_T \rightarrow L(A_T)$.
    For each $a \in A_T$, $*_{\mb A_T}(a)$ is the $<^{\mb A_T}$-least element of $\{x \in A_T: a \unlhd^{\mb A_T} x$ and $x$ is $\lhd^{\mb A_T}$-maximal in $\mb A_T\}$.
    It is straightforward to see that, for any two $\unlhd^{\mathbf A_T}$-incomparable $a, b\in A_T$,
    \begin{enumerate}
        \item $\mb A_T \models \sqcap(*_{\mb A_T}(a), *_{\mb A_T}(b)) = \sqcap(a, b)$, and
        \item $\mb A_T \models a < b \leftrightarrow *_{\mb A_T}(a) < *_{\mb A_T}(b)$.
    \end{enumerate}
\end{observation}

Recall that for $\mb A \in \mc C^<$, $\iota_{\mb A}: A \rightarrow L(T(A))$ is an isomorphism from $\mb A$ to $L(T(\mb A))$ that maps $a$ to $[a, a]$.

\begin{proposition}\label{prop:* embedding}
    Let $\mb A, \mb B \in \mc C^<$ and $\mb C_T \in \mc{CT}$.
    \begin{enumerate}
        \item For any $\alpha_T \in \emb(T(\mb A), \mb C_T)$, $\alpha^*: \mb A \rightarrow L(\mb C_T)$ defined by $\alpha^* = *_{\mb C_T} \circ \alpha_T \circ \iota_{\mb A}$ is an embedding from $\mb A$ into $L(\mb C_T)$.
        \item Let $\beta_T \in \emb(T(\mb B), \mb C_T)$ and $\alpha \in \emb(\mb A, \mb B)$.
        Then $\beta^* \circ \alpha = (\beta_T \circ T(\alpha))^* \in \emb(\mb A, L(\mb C_T))$.
    \end{enumerate}
\end{proposition}

\begin{proof}
    \begin{enumerate}
        \item By Definition \ref{def:C-rel}(1)(2)(4), if $a, b, c$ are not pairwise distinct, then $\mb A \models C(a, b, c)$ if and only if $b = c \ne a$, thus $\mb A \models C(a, b, c)$ if and only if $L(\mb C_T) \models C(\alpha^*(a), \alpha^*(b), \alpha^*(c))$.

        Now assume that $a, b, c$ are pairwise distinct and that $\mb A \models C(a,b,c)$.
        We temporarily regard $\iota_{\mb A}$ as an inclusion map for notational simplicity.
        Then since $\iota_{\mb A}$ is an isomorphism from $\mb A$ to $L(T(\mb A))$, $T(\mb A) \models \sqcap(a, b) = \sqcap(a, c) \lhd \sqcap(b, c)$.
        As $\alpha_T \in \emb(T(\mb A), \mb C_T)$, $\mb C_T \models \sqcap(\alpha_T(a), \alpha_T(b)) = \sqcap(\alpha_T(a), \alpha_T(c)) \lhd \sqcap(\alpha_T(b), \alpha_T(c))$.
        By Observation \ref{obs:tree*}, $\mb C_T \models \sqcap(\alpha^*(a)$, $\alpha^*(b)) = \sqcap(\alpha_T(a), \alpha_T(b)) = \sqcap(\alpha_T(a), \alpha_T(c)) = \sqcap(\alpha^*(a), \alpha^*(c)) \lhd \sqcap(\alpha_T(b)$, $\alpha_T(c)) = \sqcap(\alpha^*(b), \alpha^*(c))$.
        Thus $L(\mb C_T) \models C(\alpha^*(a), \alpha^*(b), \alpha^*(c))$.
        If $\mb A \models \neg C(a, b, c)$, then $L(\mb C_T) \models \neg C(\alpha^*(a), \alpha^*(b), \alpha^*(c))$ by the same argument.

        To see that $\alpha^*$ preserves $<$, note that $\iota_{\mb A}$ and $\alpha_T$ preserve $<$, and $*_{\mb C_T}$ also preserves $<$ on $\lhd^{\mb C_T}$-incomparable elements by Observation \ref{obs:tree*}.
        \item For any $a \in A$,
        \begin{align*}
            (\beta_T \circ T(\alpha))^*(a)& = (*_{\mb C_T} \circ \beta_T \circ T(\alpha) \circ \iota_{\mb A})(a)\\
            & = (*_{\mb C_T} \circ \beta_T \circ T(\alpha))([a, a])\\
            & = (*_{\mb C_T} \circ \beta_T)([\alpha(a), \alpha(a)])\\
            & = (*_{\mb C_T} \circ \beta_T \circ \iota_{\mb B} \circ \alpha)(a)\\
            & = (\beta^* \circ \alpha)(a).\qedhere
        \end{align*} 
    \end{enumerate}
\end{proof}

\begin{theorem}\label{thm:C-rel}
    Let $\mb A, \mb B \in \mc C^<$, $2 \le r \in \omega$, and $\mb C_T \in \mc{CT}$ satisfy $\mb C_T \rightarrow (T(\mb B))^{T(\mb A)}_r$.
    Such a tree $\mb C_T$ can be obtained by Theorem \ref{thm:final Ramsey witness}.
    Then $L(\mb C_T) \in \mc C^<$ satisfies $L(\mb C_T) \rightarrow (\mb B)^{\mb A}_r$
    (see Definition \ref{def:leaves of trees} for the definition of $L(\mb C_T)$).
    Consequently, the class $\mc C^<$ is a Ramsey class.
\end{theorem}

\begin{proof}
    Let $\chi: \emb(\mb A, L(\mb C_T)) \rightarrow r$.
    Define $\chi_T: \emb(T(\mb A), \mb C_T) \rightarrow r$ by $\chi_T(\alpha_T) = \chi(\alpha^*)$, where $\alpha^*$ is defined in Proposition \ref{prop:* embedding}(1).
    By the choice of $\mb C_T$, there is $\beta_T \in \emb(T(\mb B), \mb C_T)$ such that $|\chi_T(\beta_T \circ \emb(T(\mb A), T(\mb B)))| \le 1$.
    Then $\beta^* \in \emb(\mb B, L(\mb C_T))$ by Proposition \ref{prop:* embedding}(1), and for any $\alpha \in \emb(\mb A, \mb B)$, $\beta^* \circ \alpha = (\beta_T \circ T(\alpha))^*$ by Proposition \ref{prop:* embedding}(2).
    Thus $\chi(\beta^* \circ \alpha) = \chi((\beta_T \circ T(\alpha))^*) = \chi_T(\beta_T \circ T(\alpha)) \in \chi_T(\beta_T \circ \emb(T(\mb A), T(\mb B)))$.
\end{proof}

\begin{notation}\label{not:cone subtree}
    Let $\mb A_T \in \mc{CT}$ and $a, b, c \in A_T$.
    \begin{enumerate}
        \item $\cone_{\mb A_T}(a) = \{x \in A_T: a \unlhd^{\mb A_T} x\}$, and $\cone_{\mb A_T}(\mb a)$ is the substructure of $\mb A_T$ with the underlying set $\cone_{\mb A_T}(a)$.
        \item We write $\sqcap^{\mb A_T}(\cone_{\mb A_T}(a), \cone_{\mb A_T}(b)) = c$ if $\sqcap^{\mb A_T}(x, y) = c$ for any $(x, y) \in \cone_{\mb A_T}(a) \times \cone_{\mb A_T}(b)$.
    \end{enumerate}
\end{notation}

$\mb A_T \in \mc{CT}$ is called \emph{binary branching} if every vertex has at most two $\lhd^{\mb A_T}$-immediate successors.
For $\mb A_T \in \mathcal{CT}$, $\mb A_T$ is binary branching if and only if $L(\mb A_T)$ is a binary branching $C$-relation.

\begin{lemma}\label{lemma:binary transform}
    For any $\mb A_T \in \mc{CT}$, there is a binary branching tree $\bin(\mb A_T) \in \mc{CT}$ such that
    \begin{enumerate}
        \item $A_T \subseteq \bin(A_T)$ (but $\mb A_T$ is {\em not} necessarily a substructure of $\bin(\mb A_T)$),
        \item $L(A_T) = L(\bin(A_T))$,
        \item $\unlhd^{\bin(\mb A_T)} \cap A_T^2 = \unlhd^{\mb A_T}$,
        \item for any $a, b, c \in A_T$, if $\mb A_T \models \sqcap(a, b) = \sqcap(a, c) \lhd \sqcap(b, c)$, then $\mb \bin(\mb A_T) \models \sqcap(a, b) = \sqcap(a, c) \lhd \sqcap(b, c)$, and
        \item for any $a, b \in A_T$, $\mb A_T \models a < b$ if and only if $\bin(\mb A_T) \models a < b$.
    \end{enumerate}
\end{lemma}

\begin{proof}
    Let $A_T = \{v_0 <^{\mb A_T} \cdots <^{\mb A_T} v_{N-1}\}$.
    We will inductively extend $A_T$ to $\bin(A_T)$ and define the structure of $\bin(\mb A_T) \in \mc{CT}$.

    Assume that we have defined $\bin(\mb A_T)_{k-1}$.
    If $k = 0$, then let $\bin(\mb A_T)_{-1} = \mb A_T$.
    If $v_k$ has at most $2$ immediate $\lhd^{\mb A_T}$-successors, then $\bin(\mb A_T)_k = \bin(\mb A_T)_{k-1}$.
    Otherwise, enumerate immediate $\lhd^{\mb A_T}$-successors of $v_k$ by $v_{k, 0} <^{\mb A_T} \cdots <^{\mb A_T} v_{k, m-1}$, where $m \ge 3$.
    Let $\bin(A_T)_k = \bin(A_T)_{k-1} \cup \{w_{k, 2}, \ldots, w_{k, m-1}\}$, where $\bin(A_T)_{k-1}$ $\cap$ $\{w_{k, 2}, \ldots, w_{k, m-1}\} = \emptyset$.
    We put $w_{k, m} = v_k$ for convenience.
    Give the following structure to $\bin(\mb A_T)_k$:
    \begin{enumerate}[label=(\roman*)]
        \item For each $x, y \in \bin(\mb A_T)_{k-1}$, if $\bin(\mb A_T)_{k-1} \models x < y$, then $\bin(\mb A_T)_k \models x < y$.
        \item If $\sqcap^{\bin(\mb A_T)_{k-1}}(x, y) \ne v_k$, then $\sqcap^{\bin(\mb A_T)_k}(x, y) = \sqcap^{\bin(\mb A_T)_{k-1}}(x, y)$.
        \item For each $i < j < m$, $\sqcap^{\bin(\mb A_T)_k}(\cone_{\bin(\mb A_T)_k}(v_{k, i}), \cone_{\bin(\mb A_T)_k}(v_{k, j})) = w_{k, j+1}$.
        Note that by (ii), $\cone_{\bin(\mb A_T)_k}(v_{k, i}) = \cone_{\bin(\mb A_T)_{k-1}}(v_{k, i})$ for each $i < m$.
        \item $\bin(\mb A_T)_k \models v_k = w_{k, m} \lhd w_{k, m-1} \lhd \cdots \lhd w_{k, 2}$.
        \item For each $2 \le i \le m - 1$, $\sqcap^{\bin(\mb A_T)_k}(\cone_{\bin(\mb A_T)_k}(v_{k, i}), \cone_{\bin(\mb A_T)_k}(w_{k, i})) = w_{k, i+1}$.
    \end{enumerate}
    Then we can uniquely determine the rest of the structure of $\bin(\mb A_T)_k$, so that $\bin(\mb A_T)_k \in \mc{CT}$.
    Let $\bin(\mb A_T) = \bin(\mb A_T)_{N-1}$, where $v_{N-1}$ is the $<^{\mb A_T}$-maximal element of $\mb A_T$.
    By the construction, (1), (2), (3) and (5) are clear.
    Every new vertex has exactly two immediate successors, and each original vertex has at most two after its stage;
    hence $\operatorname{Bin}(\mathbf A_T)$ is binary branching.

    It remains to show (4).
    Assume that $\mb A_T \models \sqcap(a, b) = \sqcap(a, c) \lhd \sqcap(b, c)$.
    Suppose that $\bin(\mb A_T) \models \sqcap(b, c) \unlhd a$ for a contradiction.
    Then $\bin(\mb A_T) \models \sqcap^{\mb A_T}(b, c) \unlhd \sqcap(b, c) \unlhd a$ by (3), thus $\mb A_T \models \sqcap(b, c) \unlhd a$ by (3).
    But then $\mb A_T \models \sqcap(a, \sqcap(b, c)) = \sqcap(b, c)$, and it implies that $\mb A_T \not\models \sqcap(a, b) = \sqcap(a, c) \lhd \sqcap(b, c)$, a contradiction.
    Thus 
    \begin{center}
        $\bin(\mb A_T) \not \models \sqcap(b, c) \unlhd a. \quad (\dagger)$
    \end{center}
    Note that $\sqcap^{\bin(\mb A_T)}(b, c)$ and $\sqcap^{\bin(\mb A_T)}(a, b)$ are $\unlhd^{\bin(\mb A_T)}$-comparable.
    If $\bin(\mb A_T) \models \sqcap(b, c) \unlhd \sqcap(a, b)$, then $\bin(\mb A_T) \models \sqcap(b, c) \unlhd \sqcap(a, b) \unlhd a$, which is impossible by $(\dagger)$.
    Thus $\bin(\mb A_T) \models \sqcap(a, b) \lhd \sqcap(b, c)$.
    This implies that $\bin(\mb A_T) \models \sqcap(c, \sqcap(a, b)) \lhd \sqcap(b, c)$, and consequently $\bin(\mb A_T) \models \sqcap(a, b) = \sqcap(a, c) \lhd \sqcap(b, c)$.
\end{proof}

\begin{lemma}\label{lemma:to binary c-rel}
    Let $\mb A, \mb B \in \mc C^<_{\bin}$ and $\mb C_T \in \mc{CT}$.
    \begin{enumerate}
        \item Let $\bin_{\mb C_T}$ be the inclusion map from $C_T$ into $\bin(C_T)$.
        For any $\alpha_T \in \emb(T(\mb A), \mb C_T)$, $\alpha^*_{\bin}: \mb A \rightarrow L(\bin(\mb C_T))$ defined by $\alpha^*_{\bin} = \bin_{\mb C_T} \circ *_{\mb C_T} \circ \alpha_T \circ \iota_{\mb A}$ is an embedding from $\mb A$ into $L(\bin(\mb C_T))$.
        \item Let $\beta_T \in \emb(T(\mb B), \mb C_T)$ and $\alpha \in \emb(\mb A, \mb B)$.
        Then $\beta_{\bin}^* \circ \alpha = (\beta_T \circ T(\alpha))_{\bin}^* \in \emb(\mb A, L(\bin(\mb C_T)))$.
    \end{enumerate}
\end{lemma}

\begin{proof}
    \begin{enumerate}
        \item
        Assume that $\mb A \models C(a, b, c)$.
        By the proof of Proposition \ref{prop:* embedding}(1), we already know that $\mb C_T \models \sqcap(\alpha^*(a), \alpha^*(b)) = \sqcap(\alpha^*(a), \alpha^*(c)) \lhd \sqcap(\alpha^*(b), \alpha^*(c))$.
        By Lemma \ref{lemma:binary transform}(4), we have $\bin(\mb C_T) \models \sqcap(\alpha^*(a), \alpha^*(b)) = \sqcap(\alpha^*(a), \alpha^*(c)) \lhd \sqcap(\alpha^*(b), \alpha^*(c))$.
        Thus $L(\bin(\mb C_T)) \models C(\alpha^*(a), \alpha^*(b), \alpha^*(c))$.
        
        Now assume that $\mb A \models \neg C(a, b, c)$.
        If $a = b \ne c$, $a = c \ne b$ or $a = b = c$, then $L(\bin(\mb C_T)) \models \neg C(\alpha^*(a), \alpha^*(b), \alpha^*(c))$ by Definition \ref{def:C-rel}(1) and (2).
        If $b = c \ne a$, then $\mb A \models C(a, b, c)$ by Definition \ref{def:C-rel}(4), a contradiction.
        Now suppose that $a, b, c$ are pairwise distinct and $L(\bin(\mb C_T)) \models C(\alpha^*(a), \alpha^*(b), \alpha^*(c))$ for a contradiction.
        Since $\mb A$ is a binary branching $C$-relation, exactly one of $\mb A \models C(b, a, c)$ or $\mb A \models C(c, a, b)$ holds, say, $\mb A \models C(b, a, c)$.
        We have already shown that it implies that $\bin(\mb C_T) \models \sqcap(\alpha^*(b), \alpha^*(a)) = \sqcap(\alpha^*(b), \alpha^*(c)) \lhd \sqcap(\alpha^*(a), \alpha^*(c))$.
        Then we have
        $$L(\bin(\mb C_T)) \models C(\alpha^*(a), \alpha^*(b), \alpha^*(c)) \wedge C(\alpha^*(b), \alpha^*(a), \alpha^*(c)),$$
        which is impossible by Definition \ref{def:C-rel}(2).
        
        Finally, since $\alpha^*$ preserves $<$ and $\bin_{\mb C_T}$ preserves $<$ by Lemma \ref{lemma:binary transform}, $\alpha^*_{\bin}$ preserves $<$.
        Hence $\alpha_{\bin}^* \in \emb(\mb A, L(\bin(\mb C_T)))$.
        \item For any $a \in A$,
        \begin{align*}
            (\beta_T \circ T(\alpha))^*_{\bin}(a)& = (\bin_{\mb C_T} \circ *_{\mb C_T} \circ \beta_T \circ T(\alpha) \circ \iota_{\mb A})(a)\\
            & = (\bin_{\mb C_T} \circ *_{\mb C_T} \circ \beta_T \circ T(\alpha))([a, a])\\
            & = (\bin_{\mb C_T} \circ *_{\mb C_T} \circ \beta_T)([\alpha(a), \alpha(a)])\\
            & = (\bin_{\mb C_T} \circ *_{\mb C_T} \circ \beta_T \circ \iota_{\mb B} \circ \alpha)(a)\\
            & = (\beta^*_{\bin} \circ \alpha)(a).\qedhere
        \end{align*} 
    \end{enumerate}
\end{proof}

\begin{theorem}\label{thm:bin C-rel}
    Let $\mb A, \mb B \in \mc C^<_{\bin}$, $2 \le r \in \omega$, and $\mb C_T \in \mc{CT}$ satisfy $\mb C_T \rightarrow (T(\mb B))^{T(\mb A)}_r$.
    Such a tree $\mb C_T$ can be obtained by Theorem \ref{thm:final Ramsey witness}.
    Then $L(\bin(\mb C_T)) \in \mc C^<_{\bin}$ satisfies $L(\bin(\mb C_T)) \rightarrow (\mb B)^{\mb A}_r$
    (see Definition \ref{def:leaves of trees} and Lemma \ref{lemma:binary transform} for the definition of $L(\bin(\mb C_T))$).
    Consequently, the class $\mc C^<_{\bin}$ is a Ramsey class.
\end{theorem}

\begin{proof}
    Let $\chi: \emb(\mb A, L(\bin(\mb C_T))) \rightarrow r$.
    Define $\chi_T: \emb(T(\mb A), \mb C_T) \rightarrow r$ by $\chi_T(\alpha_T) = \chi(\alpha^*_{\bin})$, where $\alpha^*_{\bin}$ is defined in Lemma \ref{lemma:to binary c-rel}.
    By the choice of $\mb C_T$, there is $\beta_T \in \emb(T(\mb B), \mb C_T)$ such that $|\chi_T(\beta_T \circ \emb(T(\mb A), T(\mb B)))| \le 1$.
    Then $\beta^*_{\bin} \in \emb(\mb B, L(\bin(\mb C_T)))$ by Lemma \ref{lemma:to binary c-rel}(1) and for any $\alpha \in \emb(\mb A, \mb B)$, $\beta^*_{\bin} \circ \alpha = (\beta_T \circ T(\alpha))^*_{\bin}$ by Lemma \ref{lemma:to binary c-rel}(2).
    Thus $\chi(\beta^*_{\bin} \circ \alpha) = \chi((\beta_T \circ T(\alpha))^*_{\bin}) = \chi_T(\beta_T \circ T(\alpha)) \in \chi_T(\beta_T \circ \emb(T(\mb A), T(\mb B)))$.
\end{proof}

\section{Discussion}\label{sec:discussion}

\subsection{Failure of local finiteness}\label{sec:non locally finite}

In contrast with Theorem \ref{thm:F LFS}, the classes considered below are not locally finite subclasses of the corresponding classes of all finite linearly ordered set-valued structures.

\begin{remark}\label{remark:Not LFS}
    \begin{enumerate}
        \item $\mc{CT}_s$ in the language $\mc L_s = \{\sqcap, \lhd, <, (P_n)_{n \in \omega}\}$ is not a locally finite subclass of $\mathcal R_s$, the class of all finite $<$-linearly ordered set-valued $\mc L_s$-structures.
        \item The class $\mc{CT}_{\str}$ of finite convexly ordered treeable meet semilattices with level comparison in the language $\mc L_{\str}= \{\sqcap, \lhd, <, <_{\len}\}$ is not a locally finite subclass of the class of all finite $<$-linearly ordered set-valued $\mathcal L_{\str}$-structures.
        \item The class $\mc{CT}$ of finite convexly ordered treeable meet semilattices in the language $\mc L_0 = \{\sqcap, \lhd, <\}$ is not a locally finite subclass of the class of all finite $<$-linearly ordered set-valued $\mathcal L_0$-structures.
    \end{enumerate}
\end{remark}

\begin{proof}
    \begin{enumerate}
        \item We construct fixed structures $\mathbf B\in\mc{CT}_s$ and $\mathbf C_0\in\mathcal R_s$ such that, for every $n\in\omega$, there is a finite partial $\mathcal L_s$-structure satisfying conditions (1)--(3) of Definition \ref{def:LFS}, but having no completion in $\mc{CT}_s$.

        Let $B=\{a,b,c,d\}$, and let $\mathbf B\in\mc{CT}_s$ satisfy
        \begin{itemize}
            \item $P_0(a) \wedge P_1(b) \wedge P_1(c) \wedge P_2(d)$,
            \item $(a \lhd b) \wedge (a \lhd c \lhd d)$,
            \item $\sqcap(b,c) = a$, and
            \item $a < b < c < d$.
        \end{itemize}
        Given $n \in \omega$, choose any even integer $m > n$.
        Put
        $$C_m = \{a\} \cup \{t_i, u_i, v_i, w_i:i < m \},$$
        with $t_m = t_0$.
        Define a partial $\mathcal L_s$-structure $\mathbf C_m$ by placing copies $\mathbf B_i^0$ and $\mathbf B_i^1$ of $\mathbf B$ on the tuples
        $$(a, t_i, u_i, v_i) \text{ and } (a, w_i, t_{i + 1}, v_i),$$
        respectively, for each $i < m$.
        That is, the maps sending $(a, b, c, d)$ to these tuples are embeddings.
        The interpretations of the relation symbols and of $\sqcap$ in $\mb C_m$ are exactly the unions of their interpretations on these copies of $\mb B$.
        In particular,
        \begin{itemize}
            \item $t_i,t_{i+1}$ are $<^{\mathbf C_m}$-incomparable, as are $u_i,t_{i+1}$, and
            \item $\sqcap^{\mathbf C_m}(t_i,t_{i+1}) = \sqcap^{\mathbf C_m}(u_i,t_{i+1})=\emptyset$.
        \end{itemize}
        Every irreducible substructure of $\mathbf C_m$ is contained in one of the copies of $\mb B$, so Definition \ref{def:LFS}(2) holds.

        Suppose, towards a contradiction, that there are $\mathbf K\in\mc{CT}_s$ and a homomorphism-embedding $f:\mathbf C_m\to\mathbf K$.
        For every $i<m$,
        \begin{itemize}
            \item $f(u_i)$ and $f(t_{i+1})$ are $\lhd^{\mb K}$-predecessors of $f(v_i)$, and
            \item $f(u_i)$ and $f(t_{i+1})$ belong to $P_1^{\mathbf K}$.
        \end{itemize}
        Thus we obtain $f(u_i) = f(t_{i + 1})$.
        As $t_i<^{\mathbf C_m}u_i$, it follows that
        $$f(t_i) <^{\mathbf K} f(u_i) = f(t_{i + 1}).$$
        This yields a cycle
        $$f(t_0) <^{\mathbf K} \cdots <^{\mathbf K} f(t_{m-1}) <^{\mathbf K} f(t_0),$$
        a contradiction.
        Hence there is no completion of $\mb C_m$ in $\mc{CT}_s$.

        We next verify Definition \ref{def:LFS}(3).
        Let $\mb C' \subseteq \mb C_m$ have at most $n$ elements.
        Since $n < m$, some $t_j$ does not belong to $C'$.
        After cyclically relabeling the indices, we may assume that $t_0 \notin C'$.
        Let
        $$K = \{a\} \cup \{w_i, u_i, v_i: i < m\}.$$
        Let $P_0^{\mb K} = \{a\}$, and for each $i < m$, let
        $$\mb K \models (a \lhd w_i) \wedge (a \lhd u_i) \wedge (u_i \lhd v_i) \wedge  P_1(w_i) \wedge P_1(u_i) \wedge P_2(v_i).$$
        Order the immediate $\lhd^{\mb K}$-successors of $a$ by
        $$\mb K \models w_0 < \cdots < w_{m - 1} < u_0< \cdots < u_{m-1},$$
        and equip $\mb K$ with the induced convex linear order and its natural meet operation.
        This defines $\mb K \in \mc{CT}_s$.
        Define $g: C' \rightarrow K$ by
        $$g(t_i) = u_{i - 1} \text{ whenever }t_i \in C', \text{ and } g(x) = x \text{ for all other }x \in C'.$$
        $g$ is a homomorphism, and its restriction to each $\mb C' \cap \mb B_i^j$, for $j \in \{0, 1\}$, is an embedding.
        Every irreducible substructure of $\mb C'$ is contained in one of these intersections, so $g$ is a homomorphism-embedding.
        Thus $\mathbf K$ is a completion of $\mathbf C'$.

        It remains to construct the fixed structure $\mb C_0 \in \mc R_s$ and verify condition (1) of Definition \ref{def:LFS}.
        Let
        $$C_0 = \{a, w, t_0, t_1, u_0, u_1, v_0, v_1\}.$$
        Place copies of $\mb B$ on
        $$(a, t_i, u_i, v_i) \text{ and }(a, w, t_{i + 1}, v_i) \quad (i < 2), \text{ where }t_2 = t_0,$$
        and let
        $$\mb C_0 \models a < w < t_0 < u_0 < t_1 < v_0 < u_1 < v_1.$$
        The interpretations of the remaining relation symbols and of $\sqcap$ are exactly the unions of their interpretations in these four copies.
        Define $h_m: C_m \rightarrow C_0$ by
        \begin{itemize}
            \item $h_m(a) = a$,
            \item $h_m(w_i) = w \text{ for each } i < m$, and
            \item $h_m(t_i) = t_{i \bmod 2}, \quad h_m(u_i) = u_{i \bmod 2}, \quad h_m(v_i) = v_{i \bmod 2}.$
        \end{itemize}
        Since $m$ is even, $h_m$ maps each copy of $\mb B$ in $\mathbf C_m$ isomorphically onto a copy of $\mb B$ in $\mb C_0$.
        Hence $h_m$ is a homomorphism-embedding, as required.

        \item Let $\mb B^+_s \in \mc{CT}_s^+$ be the unique $\mc L_s^+$-expansion of $\mb B \in \mc{CT}_s$ in (1), and let $\mb B_{\str} = \mb B^+_s |_{\mc L_{\str}} \in \mc{CT}_{\str}$.
        Construct $\mb C_{m, \str}$ and $\mb C_{0, \str}$ from copies of $\mb B_{\str}$ by the same scheme as in (1).
        The same map $h_m$ verifies condition (1) of Definition \ref{def:LFS}, and every irreducible substructure of $\mb C_{m,\str}$ is still contained in a copy of $\mb B_{\str}$, verifying condition (2) of Definition \ref{def:LFS}.
        For condition (3) of Definition \ref{def:LFS}, use the small-substructure completions from (1), add their induced level comparison relation, and then forget the predicates $P_n$.

        Suppose that there are $\mb K \in \mc{CT}_{\str}$ and a homomorphism-embedding $f: \mb C_{m, \str} \to \mb K$.
        For every $i < m$, the copies of $\mb B_{\str}$ give
        \begin{itemize}
            \item $f(u_i)$ and $f(t_{i+1})$ are $\lhd^{\mb K}$-predecessors of $f(v_i)$,
            \item $\sqcap^{\mb K}(f(t_i), f(u_i)) = f(a)$,
            \item $f(t_i)$ and $f(u_i)$ are $\unlhd^{\mb K}$-incomparable, and
            \item $f(t_i) <^{\mb K} f(u_i)$.
        \end{itemize}
        By the convexity of $<^{\mb K}$, we get that $f(t_i) <^{\mathbf K} f(t_{i + 1})$ for each $i < m$.
        Again, these inequalities form a $<^{\mb K}$-cycle, a contradiction.

        \item Take the $\mathcal L_0$-reducts of $\mathbf B$, $\mathbf C_m$, and $\mathbf C_0$ from (1).
        The same arguments verify conditions (1) and (2) of Definition \ref{def:LFS}, while condition (3) of Definition \ref{def:LFS} follows by forgetting the level predicates in the small-substructure completions from (1).
        The argument in (2) uses only $\sqcap$, $\lhd$, and $<$ to show that there is no completion, so it proves that there is no completion in $\mc{CT}$. \qedhere
    \end{enumerate}
\end{proof}

\subsection{Further questions}\label{sec:Milliken}

Our results suggest asking whether the partite construction can also recover other Ramsey theorems for trees.
In this final section, we focus on the classes $\mc{DT}_m$ introduced in \cite[Section 6]{Sok15}.
We first recall their definitions.

\begin{definition}\label{def:other trees}\label{def:DT}
    Let $1 \le m \in \omega$.
    \begin{enumerate}
        \item $\mc L_m = \{\sqcap, \lhd\} \cup \{R_{m, i}: 1\leq i \leq m\}$, where each $R_{m,i}$ is a binary relation symbol.
        \item The class $\mc{DT}_m$ consists of all finite $\mc L_m$-structures $\mb A$ such that $(A,\sqcap^{\mb A},\lhd^{\mb A})$ is a treeable
        meet semilattice, every vertex has at most $m$ immediate
        $\lhd^{\mb A}$-successors, and the following conditions hold for
        $a,b,c\in A$ and $1\leq i\leq m$:
        \begin{enumerate}
            \item If $(a,b) \in R_{m, i}^{\mb A}$, then $a \lhd^{\mb A} b$.
            \item If $(a,b) \in R_{m, i}^{\mb A}$ and $b \unlhd^{\mb A} c$, then
            $(a, c) \in R_{m, i}^{\mb A}$.
            \item If $(a, b) \in R_{m, i}^{\mb A}$ and
            $\sqcap^{\mb A}(b, c) = a$, then $(a,c) \notin R_{m,i}^{\mb A}$.
            \item If $a \lhd^{\mb A} b$, then there is exactly one $i\in\{1,\ldots,m\}$ such that $(a,b) \in R_{m,i}^{\mb A}$.
        \end{enumerate}
    \end{enumerate}

\end{definition}

Equivalently, the immediate $\lhd$-successors of each vertex are assigned distinct branch labels from $\{1,\ldots,m\}$.
For $a \lhd b$, $R_{m,i}(a,b)$ holds if and only if the immediate $\lhd$-successor of $a$ on the $\lhd$-chain from $a$ to $b$ has label $i$.
Not every label needs to occur at every vertex, and the language contains no linear order.
We included the tree order $\lhd$ explicitly, although it is quantifier-free
definable from the meet operation $\sqcap$.

\begin{fact}[{\cite[Section 6]{Sok15}}]\label{fact:DT-Ramsey}
    For every integer $m \ge 1$, the class $\mc{DT}_m$ is a \fr class and has the Ramsey property.
\end{fact}

For $m=2$, the two branch labels can instead be encoded by placing the descendants in one direction before their ancestor and the descendants in the other direction after the ancestor.
This gives the following class $\mc{BT}_2$ from \cite{Sok15}.

\begin{definition}\label{def:BT}
    The class $\mc{BT}_2$ consists of all finite
    $\mc L_0$-structures $\mb A=(A$, $\sqcap^{\mb A}$, $\lhd^{\mb A}$, $<^{\mb A})$ such that
    $(A,\sqcap^{\mb A},\lhd^{\mb A})$ is a treeable meet semilattice in which every vertex
    has at most two immediate $\lhd^{\mb A}$-successors, and $<^{\mb A}$ is a linear order
    satisfying the following conditions:
    \begin{enumerate}
        \item If $a$ has two distinct immediate $\lhd^{\mb A}$-successors
        $b_0,b_1$, then, after possibly interchanging $b_0$ and $b_1$,
        $\mb A \models x < a < y$ for any $x \unrhd^{\mb A} b_0$ and $y \unrhd^{\mb A} b_1$.
        \item If $\mb A \models a \lhd b$, then either $\mb A \models x < a$ for every $x \unrhd^{\mb A} b$, or
        $\mb A \models a < x$ for every $x \unrhd^{\mb A} b$.
    \end{enumerate}
    In this definition, $<^{\mb A}$ is not required to extend $\lhd^{\mb A}$.
\end{definition}

\begin{remark}\label{remark:Sokic-conventions}
    The definition of $\mc{DT}_m$ in \cite[Section 6]{Sok15} omitted the uniqueness of a label $R_{m, i}$ for each pair $(a,b)$ with $a \lhd b$.
    We include uniqueness in Definition~\ref{def:DT}(2)(d).
    Without this uniqueness, even amalgamation can fail when $m \ge 2$.
    
    Likewise, the definition of $\mc{BT}_2$ in \cite{Sok15}
    imposes its $<$-order conditions only at vertices with two immediate
    $\lhd$-successors.
    We also impose Definition~\ref{def:BT}(2),
    including at vertices with a single immediate $\lhd$-successor.
\end{remark}

The bidefinability of $\mc{BT}_2$ and $\mc{DT}_2$ noted in \cite{Sok15} can be expressed by quantifier-free formulas.
Starting with $\mb A\in\mc{BT}_2$, define
\begin{align*}
    R_{2,1}(x,y) & \quad \Longleftrightarrow \quad x \lhd y \wedge y < x,\\
    R_{2,2}(x,y) & \quad \Longleftrightarrow \quad x \lhd y \wedge x < y.
\end{align*}
Conversely, for a structure in $\mc{DT}_2$, define
\begin{align*}
    x < y \quad \Longleftrightarrow \quad R_{2, 2}(x,y) \vee R_{2, 1}(y, x) \vee (R_{2, 1}(\sqcap(x, y), x) \wedge R_{2, 2}(\sqcap(x, y), y)).
\end{align*}
The two constructions are inverse to one another and, being quantifier-free, preserve embeddings in both directions.
Consequently, Ramsey witnesses transfer between these two classes.

For any $1 \le m \in \omega$, there is a definable canonical convex linear order $<_{\operatorname{lex}}$ on each member of $\mc{DT}_m$, defined by
\begin{align*}
    x <_{\mathrm{lex}} y \quad \Longleftrightarrow \quad x \lhd y \vee \bigvee_{1 \leq i < j \leq m} (R_{m, i}(\sqcap(x, y),x) \wedge R_{m, j}(\sqcap(x, y), y)).
\end{align*}
Formally adding this $<_{\operatorname{lex}}$ to $\mc L_m$ as a relational symbol results in a definitional expansion of the \fr limit of $\mc{DT}_m$, and hence does not change the embeddings or the Ramsey property.
For $m=2$, this convex order should not be confused with the order used
in $\mc{BT}_2$ above.

The case $m = 1$ is the classical Ramsey theorem for finite linear orders.
For $m \ge 2$, Soki\'c's proof proceeds by induction on the size of the structure whose copies are colored, using the Ramsey theorem for strong subtrees in the one-vertex base case and a product Ramsey theorem in the inductive step.
One of the consequences of Milliken's tree theorem \cite{Mil79} relevant here says that every finite coloring of the strong subtrees with a fixed finite number of levels in the infinite full $m$-ary tree is constant on all such strong subtrees of some infinite strong subtree.
Here each level of a strong subtree lies in a single level of the ambient tree, and for each pair of consecutive selected levels, every vertex of the subtree on the first has exactly one vertex of the subtree on the second in each direction determined by an immediate $\lhd$-successor in the ambient tree.
The finite form of the one-level case of this corollary yields the one-vertex base case in the proof of \cite[Theorem 6.1]{Sok15};
we can ask whether the Ramsey property of $\mathcal{DT}_m$ can be proved by the partite construction without using a tree-specific Ramsey theorem.

The question below allows auxiliary expansions and subsequent Ramsey transfer, as in Section \ref{sec:partite};
it is not restricted to a direct application of the locally finite subclass theorem in the original language $\mc L_m$.
Indeed, the obstruction in Subsection \ref{sec:non locally finite} and the expansion used in Section \ref{sec:partite} show why this distinction matters even for $\mc{CT}$.
Thus we seek a partite construction approach that shows the Ramsey property of $\mc{DT}_m$, possibly through auxiliary expansions, rather than using a Ramsey theorem for strong subtrees as an external input.

\begin{question}\label{question:DT-partite}
    For each integer $m\geq2$, is there a proof of the Ramsey property of
    $\mc{DT}_m$ based on the partite construction?
\end{question}

\section*{Acknowledgments}

The author is particularly grateful to Chang‐Yeon Chough for providing the opportunity to continue pursuing independent research.

The author would like to thank Joonkyung Lee for giving the author opportunities to learn about a wide range of topics in combinatorics. The author is also grateful to Jisun Baek and Jaehyeon Seo for studying \cite{Bod15} and \cite{HK26} together in a seminar over several months in 2025. Special thanks are due to Jan Hubička and Matěj Konečný for their warm hospitality at the Midsummer Combinatorics Workshop and for generously spending many hours answering questions about \cite{HK26}. Finally, the author thanks Samuel Braunfeld for the invitation to the Midsummer Combinatorics Workshop, which encouraged the author to continue exploring structural Ramsey theory.

\section*{Funding}

This research was supported by Global - Learning \& Academic research institution for Master’s·PhD students, and Postdocs(LAMP) Program of the National Research Foundation of Korea(NRF) grant funded by the Ministry of Education(No. RS-2024-00441954), and partially supported by the Research Institute for Mathematical Sciences, an International Joint Usage/Research Center located in Kyoto University.

\section*{AI Declaration}

The author used ChatGPT (OpenAI; GPT-5.5, GPT-5.6 Sol and GPT-6) to discuss proof strategies, review mathematical arguments, identify relevant literature, and assist with drafting and revising the manuscript.
In particular, ChatGPT suggested the idea of transferring Ramsey witnesses from $\mc{CT}$ to $\mc C^<$ (Section \ref{sec:C-rel}).
Other than this suggestion, the proof strategies in the paper were designed by the author, with ChatGPT assisting in working out the details.
No suggestion was incorporated into the manuscript without meticulous assessment and verification by the author.
The author takes full responsibility for the final manuscript.

\printbibliography

\end{document}